\documentclass[12pt,reqno]{amsart}
\usepackage{amsmath,amsfonts,amssymb,amsthm,amsbsy,latexsym,cite,amscd}
\usepackage[cp1251]{inputenc}
\usepackage[T2A]{fontenc}
\usepackage[colorlinks,breaklinks,linkcolor=blue,citecolor=blue,menucolor=blue]{hyperref}

\multlinegap=\parindent

\begin{document}

\newtheorem*{etheorem}{Theorem}
\newtheorem*{ecorollary}{Corollary}
\newtheorem{eproposition}{Proposition}[section]

\theoremstyle{definition}
\newtheorem{eexample}{Example}

\renewcommand{\theequation}{$\mathfrak{C}_{e}$}
\renewcommand{\thefootnote}{}

\title[On~the~residual polynilpotence of~tree products of~groups]{On~the~residual polynilpotence of~tree products of~groups with~locally cyclic edge subgroups}

\author{E.~V.~Sokolov}
\address{Ivanovo State University, Russia}
\email{ev-sokolov@yandex.ru}

\begin{abstract}
Let $\mathcal{N}$ be~a~non-empty class of~nilpotent groups closed under~taking subgroups, quotient groups, finite direct products, and~unrestricted direct powers of~the~form $\prod_{x \in X}Y_{x}$, where $X \in \mathcal{N}$ and~$Y_{x}$ ($x \in X$) are~isomorphic copies of~a~group $Y \in \mathcal{N}$. Suppose also that $T$ is~a~tree product of~$n$~groups ($2 \leqslant n < \infty$) with~non-triv\-i\-al tor\-sion-free locally cyclic edge subgroups, $\mathcal{N}^{\mathfrak{tf}}$~is~the~class of~all tor\-sion-free $\mathcal{N}$\nobreakdash-groups, and~the~symbol $\mathit{Poly}_{k}\mathcal{N}^{\mathfrak{tf}}$ ($k \geqslant 1$) denotes the~class of~polynilpotent groups such that $X \in \mathit{Poly}_{k}\mathcal{N}^{\mathfrak{tf}}$ if~and~only~if~$X$ has a~subnormal series of~length at~most~$k$ whose factors belong to~$\mathcal{N}^{\mathfrak{tf}}$. We~prove that if~all free and~all vertex groups are~residually $\mathcal{N}^{\mathfrak{tf}}$\nobreakdash-groups, then $T$ is~residually a~$(\mathit{Poly}_{n+1}\mathcal{N}^{\mathfrak{tf}})$-group. We~prove also that if~all vertex groups lie in~$\mathcal{N}$, then $T$ has a~homomorphism onto~a~$(\mathit{Poly}_{n-1}\mathcal{N}^{\mathfrak{tf}})$-group acting injectively on~all edge subgroups. By~assuming that $\mathcal{N}$ is~the~class of~all nilpotent groups, we~get a~stronger result: if~all vertex groups are~residually $\mathcal{N}^{\mathfrak{tf}}$\nobreakdash-groups, then $T$ is~residually a~$(\mathit{Poly}_{n}\mathcal{N}^{\mathfrak{tf}})$-group. It~follows that all cy\-cli\-cal\-ly-pinched one-re\-la\-tor groups are~residually finitely generated $(\mathit{Poly}_{2}\mathcal{N}^{\mathfrak{tf}})$-groups.
\end{abstract}

\keywords{Residual properties, residual nilpotence, residual polynilpotence, residual solvability, tree product, generalized free product, HNN-ex\-tens\-ion, pre-root class of~groups}

\thanks{The~study was~carried~out within the~framework of~the~grant No.~2026\nobreakdash-03 supported by~Ivanovo State University}

\vspace*{8pt}\maketitle\vspace{-14pt}

\section{Introduction}\label{esec01}

Let\footnote{This is the~English translation of~the~Russian original article given below.} $\mathcal{C}$ be~a~class of~groups (here and~below, it~is~assumed that all classes of~groups under consideration are~non-empty). Recall that a~group~$X$ is~said to~be~\emph{residually a~$\mathcal{C}$\nobreakdash-group} if~any of~its non-triv\-i\-al elements can~be~mapped to~a~non-triv\-i\-al one by~a~homomorphism of~$X$ onto~a~\emph{$\mathcal{C}$\nobreakdash-group} (i.e.,~a~group from~$\mathcal{C}$)~\cite{Hall1954PLMS}. Let~us denote by~$\mathcal{C}^{*}(X)$ the~family of~the~kernels of~all such homomorphisms. If~$\mathcal{C}$ is~a~class of~finite, nilpotent, polynilpotent, or~solvable groups, then one says of~\emph{residual finiteness}, \emph{residual nilpotence}, \emph{residual polynilpotence}, or~\emph{residual solvability} of~a~group, respectively. Following this logic,\linebreak we will~use the~term \emph{residual $\mathcal{C}$\nobreakdash-ness} in~the~general case.

Let~us recall that, in~accordance with~one of~the~equivalent definitions given in~\cite{Sokolov2015CA}, a~class of~groups~$\mathcal{C}$ is~referred to~as~a~\emph{root class} if~it~is~closed under taking subgroups, extensions, and~unrestricted direct powers of~the~form $\prod_{x \in X}Y_{x}$, where $X \in \mathcal{C}$ and~$Y_{x}$ ($x \in X$) are~isomorphic copies of~a~group $Y \in \mathcal{C}$. The~properties of~$\mathcal{C}$ listed in~this definition have~allowed one to~make significant progress in~studying the~residual $\mathcal{C}$\nobreakdash-ness of~free constructions of~groups: free and~tree products, HNN-ex\-tens\-ions, fundamental groups of~graphs of~groups,~etc. (see, e.g.,~\cite{AzarovTieudjo2002PIvSU, Tumanova2019SMJ, Sokolov2021SMJ1, Sokolov2021SMJ2, Sokolov2022CA, SokolovTumanova2026JGT, SokolovTumanova2026SMJ}). When $\mathcal{C}$ is~a~\emph{root class}, the~residual $\mathcal{C}$\nobreakdash-ness of~groups has~been most~actively studied in~the~literature in~cases where $\mathcal{C}$ is~the~class~of

--\hspace{1ex}all finite groups,

--\hspace{1ex}finite $p$\nobreakdash-groups (where $p$ is~a~prime),

--\hspace{1ex}finite $\mathfrak{P}$\nobreakdash-groups (where $\mathfrak{P}$ is~a~non-empty set of~primes),

--\hspace{1ex}periodic $\mathfrak{P}$\nobreakdash-groups of~finite exponent,

--\hspace{1ex}all solvable groups,~or

--\hspace{1ex}tor\-sion-free solvable groups.

\noindent
At~the~same time, it~is~easy to~give an~example of~a~\emph{non-root class}~$\mathcal{C}$ such that the~property of~residual $\mathcal{C}$\nobreakdash-ness is~of~significant interest. First of~all, we are~referring to~the~class of~all nilpotent groups and~its various subclasses. Most~of~them are~not~closed under taking extensions, but~satisfy the~weaker property of~being~closed under finite direct products (i.e.,~direct products with~a~finite number of~factors). This last observation is~important, because if~a~class~$\mathcal{C}$ contains subgroups and~finite direct products of~its groups, then every group~$X$ can~be~equipped with~the~so-called pro-$\mathcal{C}$-topology whose neighborhood basis of~the~identity is~the~family~$\mathcal{C}^{*}(X)$ defined above. Pro-$\mathcal{C}$-topologies play an~important role in~the~study of~residual properties of~free constructions of~groups, and~this suggests the~introduction of~the~following concept.

We will~say that a~class of~groups~$\mathcal{C}$ is~a~\emph{pre-root class} if~it~is~closed under taking subgroups, finite direct products, and~unrestricted direct powers of~the~form $\prod_{x \in X}Y_{x}$, where $X \in \mathcal{C}$ and~$Y_{x}$ ($x \in X$) are~isomorphic copies of~a~group $Y \in \mathcal{C}$. It~follows directly from~the~definitions of~root and~pre-root classes that

1)\hspace{1ex}any variety of~groups is~a~pre-root class;

2)\hspace{1ex}the~intersection of~any number of~root and~pre-root classes is~again a~pre-root class (in~particular, if~$\mathcal{C}$ is~a~pre-root class, then the~class of~all tor\-sion-free $\mathcal{C}$\nobreakdash-groups has the~same property);

3)\hspace{1ex}a~class consisting of~finite groups is~a~pre-root class if~and~only~if~it~is~closed under taking subgroups and~finite direct products.

Repeating word for~word the~proof of~Proposition~17 from~\cite{SokolovTumanova2020IVM}, one can also show that if~a~pre-root class~$\mathcal{C}$ consists of~periodic groups, then every $\mathcal{C}$\nobreakdash-group is~of~finite exponent. The~examples of~pre-root classes are~the~classes~of

--\hspace{1ex}all nilpotent groups,

--\hspace{1ex}finite nilpotent groups,

--\hspace{1ex}tor\-sion-free nilpotent groups,

--\hspace{1ex}periodic nilpotent $\mathfrak{P}$\nobreakdash-groups of~finite exponent (where $\mathfrak{P}$ is~a~non-empty set of~primes),

--\hspace{1ex}nilpotent groups whose nilpotency classes (exponents) do~not~exceed (respectively,{\parfillskip=0pt\par}

\mbox{}\phantom{--}\hspace{1ex}divide) some given number,

\noindent
and all possible intersections of~these classes. The~following construction serves as~another source of~examples of~pre-root classes.

For~any class of~groups~$\mathcal{C}$ and~for~any number $n \geqslant 1$, let~$\mathit{Poly}_{n}\mathcal{C}$ denote the~class of~groups defined by~the~rule: $X \in \mathit{Poly}_{n}\mathcal{C}$ if~and~only~if~$X$ has a~subnormal series of~length at~most~$n$ whose factors belong to~$\mathcal{C}$. Clearly, if~$\mathcal{C}$ is~closed under taking extensions, then $\mathcal{C} = \mathit{Poly}_{n}\mathcal{C}$ for~each $n \geqslant 1$. It~is~easy to~see that the~class~$\mathit{Poly}_{n}\mathcal{C}$ is~closed under taking quotient groups and/or~consists of~finite groups, periodic groups, or~tor\-sion-free groups if~and~only~if~$\mathcal{C}$ has the~same property. Let~us note also that if~$\mathcal{C}$ is~closed under taking finite direct products, then every $(\mathit{Poly}_{n}\mathcal{C})$\nobreakdash-group is~equinumerous to~some $\mathcal{C}$\nobreakdash-group. Namely, if~$X \in \mathit{Poly}_{n}\mathcal{C}$ and~$1 = X_{0} \leqslant X_{1} \leqslant \ldots \leqslant X_{n} = X$ is~a~subnormal series of~$X$ with~$\mathcal{C}$\nobreakdash-fac\-tors, then $X$ is~equinumerous to~the~direct product of~$\mathcal{C}$\nobreakdash-groups~$X_{i}/X_{i-1}$ ($1 \leqslant i \leqslant n$) and~this product, in~turn, belongs to~$\mathcal{C}$. It~easily follows that if~$\mathcal{C}$ is~a~pre-root class, then $\mathit{Poly}_{n}\mathcal{C}$ has the~same property. Proposition~\ref{ep32} below also shows that if~$\mathcal{C}$ is~a~pre-root class, then the~word ``subnormal'' in~the~definition of~the~class~$\mathit{Poly}_{n}\mathcal{C}$ can~be~replaced by~the~word ``normal''.

Let $\mathcal{N}$ be~a~pre-root class of~nilpotent groups. In~this paper, we study the~residual $(\mathit{Poly}_{n}\mathcal{N})$-ness of~tree products of~groups (the~definition of~a~tree product is~given in~the~next section). There are~very few results on~the~residual $\mathcal{C}$\nobreakdash-ness of~this construction, where $\mathcal{C}$ is~a~non-root class of~nilpotent or~polynilpotent groups. In~the~case of~a~two-vertex tree, a~comprehensive review of~such results can~be~found in~\cite{Sokolov2024JA}. The~latter paper also gives a~number of~necessary and~sufficient conditions for~a~tree product of~two groups to~be~residually a~finite nilpotent or~finite metanilpotent group. Besides these results, there seems to~be~only two criteria for~the~residual nilpotence of~the~fundamental groups of~graphs of~groups. One of~them is~for~an~arbitrary graph of~finite groups~\cite[Theorem~6]{Varsos1996HJM} and~the~other is~for~a~finite graph of~infinite cyclic groups~\cite[Theorems~5 and~6]{Sokolov2021JA}. The~present paper supplements this list with~a~sufficient condition of~residual $(\mathit{Poly}_{n}\mathcal{N})$-ness for~a~tree product of~a~finite number of~groups with~locally cyclic edge subgroups.{\parfillskip=0pt{}\par}

\section{Formulation of~results}\label{esec02}

Let $\mathcal{T}$ be~a~non-empty tree with~a~vertex set~$\mathcal{V}$ and~an~edge set~$\mathcal{E}$, and~let $\mathcal{G}(\mathcal{T})$ be~the~\emph{graph of~groups} over~$\mathcal{T}$ that is~defined by~assigning to~each vertex $v \in \mathcal{V}$ a~group~$G_{v}$, while to~each edge $e \in \mathcal{E}$ a~group~$H_{e}$ and~injective homomorphisms $\varphi_{+e}\colon H_{e} \to G_{e(1)}$ and~$\varphi_{-e}\colon H_{e} \to G_{e(-1)}$ (where $e(1)$ and~$e(-1)$ are~the~vertices that are~the~ends of~$e$). All these designations are~assumed to~be~fixed throughout the~paper.

Recall that, in~accordance with~\cite{KarrassSolitar1970TAMS}, the~fundamental group~$\pi_{1}(\mathcal{G}(\mathcal{T}))$ of~the~graph~$\mathcal{G}(\mathcal{T})$ is~said to~be~the~\emph{tree product of~the~groups~$G_{v}$} ($v \in \mathcal{V}$). Recall also that this fundamental group can~be~defined by~the~presentation whose generators are~the~generators of~$G_{v}$ ($v \in \mathcal{V}$) and~whose defining relations are~those of~$G_{v}$ ($v \in \mathcal{V}$) and~all possible relations of~the~form $h\varphi_{+e} = h\varphi_{-e}$ ($e \in \mathcal{E}$, $h \in H_{e}$), where $h\varphi_{\varepsilon e}$ ($\varepsilon = \pm 1$) is~a~word in~the~generators of~$G_{e(\varepsilon)}$ that defines the~image of~$h$ under $\varphi_{\varepsilon e}$. Let~us refer to~the~groups~$G_{v}$ ($v \in \mathcal{V}$) as~the~\emph{vertex groups}, while to~the~subgroups $H_{+e} = H_{e}\varphi_{+e}$ and~$H_{-e} = H_{e}\varphi_{-e}$ ($e \in \mathcal{E}$) as~the~\emph{edge subgroups}. It~is~known~\cite[Theorem 1]{KarrassSolitar1970TAMS} that, given a~vertex group~$G_{v}$ ($v \in \mathcal{V}$), the~identity mapping of~its generators to~$\pi_{1}(\mathcal{G}(\mathcal{T}))$ defines an~embedding. Therefore, $G_{v}$~can~be~considered as~a~subgroup of~$\pi_{1}(\mathcal{G}(\mathcal{T}))$. The~main result of~the~paper is~the~next

\begin{etheorem}
Let the~symbols~$\mathcal{N}$ and~$\mathcal{N}^{\mathfrak{tf}}$ denote a~pre-root class of~nilpotent groups closed under taking quotient groups and~the~class of~all tor\-sion-free $\mathcal{N}$\nobreakdash-groups\textup{,} respectively. If~$\mathcal{T}$ is~a~finite tree with~$n \geqslant 2$ vertices and~each subgroup~$H_{\varepsilon e}$ \textup{(}$e \in \mathcal{E}$\textup{,} $\varepsilon = \pm 1$\textup{)} is~a~non-triv\-i\-al tor\-sion-free locally cyclic group\textup{,} then the~following statements hold.

$(\alpha)$\hspace{1ex}If $G_{v} \in \mathcal{N}$ for~any $v \in \mathcal{V}$\textup{,} then $\pi_{1}(\mathcal{G}(\mathcal{T}))$ has a~homomorphism onto~a~$(\mathit{Poly}_{n-1}\mathcal{N}^{\mathfrak{tf}})$-group acting injectively on~all subgroups~$H_{\varepsilon e}$ \textup{(}$e \in \mathcal{E}$\textup{,} $\varepsilon = \pm 1$\textup{)}.

$(\beta)$\hspace{1ex}Suppose that\textup{,} for~any $e \in \mathcal{E}$ and~$\varepsilon = \pm 1$\textup{,} there exists a~homomorphism of~$G_{e(\varepsilon)}$ onto~an~$\mathcal{N}$\nobreakdash-group that acts injectively on~$H_{\varepsilon e}$. If~all~$G_{v}$ \textup{(}$v \in \mathcal{V}$\textup{)} and~all free groups are residually $\mathcal{N}$\nobreakdash-groups\textup{,} then $\pi_{1}(\mathcal{G}(\mathcal{T}))$ is~residually a~$(\mathit{Poly}_{n+1}\mathcal{N})$-group.

$(\gamma)$\hspace{1ex}If all~$G_{v}$ \textup{(}$v \in \mathcal{V}$\textup{)} and~all free groups are~residually $\mathcal{N}^{\mathfrak{tf}}$\nobreakdash-groups\textup{,} then $\pi_{1}(\mathcal{G}(\mathcal{T}))$ is~residually a~$(\mathit{Poly}_{n+1}\mathcal{N}^{\mathfrak{tf}})$-group.
\end{etheorem}

Let~us note that the~results similar to~Statement~$(\alpha)$ play an~important role in~studying the~residual properties of~tree products of~group, allowing this construction to~be~represented as~a~group extension with~well-un\-der\-stood properties. The~upper part of~such an~extension is~a~group from~a~predetermined class; it hides almost all the~complexities associated with~the~graph of~groups. The~lower part splits as~the~free product of~a~free group and~groups that are~conjugate to~subgroups of~the~vertex groups (see Proposition~\ref{ep46} below); for~this construction, most residual properties have~been fully studied.{\parfillskip=0pt{}\par}

Let~us note also that a~polynilpotent group is~often described by~the~row $(c_{1},\,\ldots,\,c_{n})$, where $c_{1}$,~\ldots,~$c_{n}$ are~the~nilpotency classes of~the~factors of~its polynilpotent series. If~we take the~class~$\mathcal{N}_{c}$ of~all nilpotent groups of~nilpotency class at~most~$c$ as~$\mathcal{N}$, then the~row $(d_{1},\,\ldots,\,d_{n-1})$, where $d_{1},\,\ldots,\,d_{n-1} \leqslant c$, corresponds to~the~image of~the~homomorphism whose existence is~guaranteed by~Statement~$(\alpha)$. But~such a~choice of~$\mathcal{N}$ is~not~possible for~Statements~$(\beta)$ and~$(\gamma)$ because there always exists a~free group that is~not~residually an~$\mathcal{N}_{c}$\nobreakdash-group. Indeed, the~class~$\mathcal{N}_{c}$ is~a~variety of~groups, and~therefore any residually $\mathcal{N}_{c}$\nobreakdash-group belongs to~$\mathcal{N}_{c}$, as~is~easily seen. Thus, Statements~$(\beta)$ and~$(\gamma)$ limit only the~lengths of~polynilpotent series of~homomorphic images of~$\pi_{1}(\mathcal{G}(\mathcal{T}))$, but~not~the~nilpotency classes of~their factors.

\begin{ecorollary}
If $\mathcal{N}^{\mathfrak{tf}}$ is~the~class of~all tor\-sion-free nilpotent groups\textup{,} $\mathcal{T}$~is~a~finite tree with $n \geqslant 2$ vertices\textup{,} and~each subgroup~$H_{\varepsilon e}$ \textup{(}$e \in \mathcal{E}$\textup{,} $\varepsilon = \pm 1$\textup{)} is~a~non-triv\-i\-al tor\-sion-free locally cyclic group\textup{,} then the~following statements hold.

$(\dag)$\hspace{1ex}If~all~$G_{v}$ \textup{(}$v \in \mathcal{V}$\textup{)} are~residually $\mathcal{N}^{\mathfrak{tf}}$\nobreakdash-groups\textup{,} then $\pi_{1}(\mathcal{G}(\mathcal{T}))$ is~residually a~$(\mathit{Poly}_{n}\mathcal{N}^{\mathfrak{tf}})$-group.

$(\ddag)$\hspace{1ex}If~all~$G_{v}$ \textup{(}$v \in \mathcal{V}$\textup{)} are free groups\textup{,} then $\pi_{1}(\mathcal{G}(\mathcal{T}))$ is~residually a~finitely generated $(\mathit{Poly}_{n}\mathcal{N}^{\mathfrak{tf}})$-group. In~particular\textup{,} all cy\-cli\-cal\-ly-pinched one-re\-la\-tor groups \textup{(}i.e.\textup{,}~free products of~two free groups with~a~cyclic amalgamation\textup{)} are~residually finitely generated $(\mathit{Poly}_{2}\mathcal{N}^{\mathfrak{tf}})$-groups.
\end{ecorollary}

In~Theorem and~Corollary, the~lengths of~polynilpotent series of~homomorphic images of~$\pi_{1}(\mathcal{G}(\mathcal{T}))$ are~bounded by~the~numbers $n-1$, $n$, and~$n+1$. It~is~generally an~open question whether these estimates can be~improved. Proposition~\ref{ep48} below and~the~main results of~\cite{BaumslagG2010IJM, Labute2015Arxiv} show that such an~improvement is~possible under~certain additional restrictions. At~the~same time, it~is~known that a~cy\-cli\-cal\-ly-pinched one-re\-la\-tor group need~not~be~residually nilpotent~\cite{Azarov1998MN}. Therefore, when $n=2$, the~upper bound given in~Corollary is~sharp.

If we forget about polynilpotence and~assume that Theorem and~Corollary deal with~homomorphisms onto~solvable groups, then it~is~natural to~ask whether their statements hold for~an~infinite tree~$\mathcal{T}$. A~negative answer to~this question is~given~by

\begin{eexample}\label{eex01}
Suppose that $\mathcal{T}$ is~an~infinite chain, its~vertices are~numbered by~integers, and~for~each $i \in \mathbb{Z}$, the~free group~$F_{i}$ with~a~basis $\{a_{i},b_{i}\}$ corresponds to~the~vertex with~the~number~$i$. Suppose also that, given $i \in \mathbb{Z}$, the~edge connecting the~vertices with~the~numbers~$i$ and~$i+1$ is~associated with~an~infinite cyclic group and~mappings sending its~generator to~the~elements $[b_{i}^{\vphantom{1}}, a_{i}^{-1}b_{i}^{\vphantom{1}}a_{i}^{\vphantom{1}}]$ and~$b_{i+1}^{\vphantom{1}}$. In~this case, $\pi_{1}(\mathcal{G}(\mathcal{T}))$ has the~presentation $\langle a_{i}^{\vphantom{1}}, b_{i}^{\vphantom{1}};\ b_{i+1}^{\vphantom{1}} = [b_{i}^{\vphantom{1}}, a_{i}^{-1}b_{i}^{\vphantom{1}}a_{i}^{\vphantom{1}}]\ (i \in \mathbb{Z}) \rangle$. It~easily follows that, for~any $i \in \mathbb{Z}$ and~$k \geqslant 1$, the~element~$b_{i+k}$ belongs to~the~$k$\nobreakdash-th~member~$(\pi_{1}(\mathcal{G}(\mathcal{T})))^{(k)}$ of~the~derived series of~$\pi_{1}(\mathcal{G}(\mathcal{T}))$. Since $b_{i} = b_{(i-k)+k}$ for~all $i \in \mathbb{Z}$ and~$k \geqslant 1$, the~inclusion $b_{i} \in \bigcap_{k \geqslant 1}(\pi_{1}(\mathcal{G}(\mathcal{T})))^{(k)}$ holds for~each $i \in \mathbb{Z}$. Thus, any homomorphism of~$\pi_{1}(\mathcal{G}(\mathcal{T}))$ onto~a~solvable group maps all edge subgroups to~the~trivial element. In~particular, $\pi_{1}(\mathcal{G}(\mathcal{T}))$ is~not~residually solvable. Let~us note also that~if~$i \in \mathbb{Z}$~and
\[
N_{i}^{\vphantom{1}} = \langle 
a_{i}^{\vphantom{1}}, b_{i}^{\vphantom{1}}, 
c_{i}^{\vphantom{1}}, d_{i}^{\vphantom{1}};\ 
a_{i}^{-1}b_{i}^{\vphantom{1}}a_{i}^{\vphantom{1}} = 
b_{i}^{\vphantom{1}}c_{i}^{\vphantom{1}},\ 
b_{i}^{-1}c_{i}^{\vphantom{1}}b_{i}^{\vphantom{1}} = 
c_{i}^{\vphantom{1}}d_{i}^{\vphantom{1}},\ 
[a_{i}^{\vphantom{1}},c_{i}^{\vphantom{1}}] = 
[a_{i}^{\vphantom{1}},d_{i}^{\vphantom{1}}] = 
[b_{i}^{\vphantom{1}},d_{i}^{\vphantom{1}}] = 
[c_{i}^{\vphantom{1}},d_{i}^{\vphantom{1}}] = 1
\rangle,
\]
then $1 < \langle d_{i} \rangle < \langle c_{i},d_{i} \rangle < \langle b_{i},c_{i},d_{i} \rangle < N_{i}$ is~a~central series of~$N_{i}$ with~infinite cyclic factors and~$[b_{i}^{\vphantom{1}}, a_{i}^{-1}b_{i}^{\vphantom{1}}a_{i}^{\vphantom{1}}] = [b_{i}^{\vphantom{1}}, b_{i}^{\vphantom{1}}c_{i}^{\vphantom{1}}] = [b_{i}^{\vphantom{1}}, c_{i}^{\vphantom{1}}] = d_{i}^{-1} \ne 1$. Hence, we can replace~$F_{i}$ ($i \in \mathbb{Z}$) with~$N_{i}$ ($i \in \mathbb{Z}$) in~the~tree product constructed above, thus proving that the~finiteness of~$\mathcal{T}$ in~a~possible analogue of~Statement~$(\alpha)$ is~essential even in~the~case where all vertex groups have the~same nilpotency class.
\end{eexample}

Theorem and~Corollary cannot be~extended to~the~construction of~HNN-ex\-tens\-ion, as~shown~by

\begin{eexample}\label{eex02}
Let $B_{0}$ denote the~group~$F_{0}$ or~$N_{0}$ from~the~previous example, and~let~$E$ be~the~HNN-ex\-tens\-ion of~$B_{0}$ with~a~stable letter~$t$ and~the~cyclic associated subgroups generated by~the~elements~$b_{0}^{\vphantom{1}}$ and~$[b_{0}^{\vphantom{1}}, a_{0}^{-1}b_{0}^{\vphantom{1}}a_{0}^{\vphantom{1}}]$. Using the~Reidemeister--Schreier method (see, e.g.,~\cite[Section~2.3]{MagnusKarrasSolitar1976}), it~is~easy to~show that the~normal closure of~$B_{0}$ in~$E$ is~isomorphic to~the~tree product from~Example~\ref{eex01} and~this isomorphism maps the~associated subgroups of~$E$ onto~the~edge subgroups of~the~tree product. It~follows that the~kernel of~any homomorphism of~$E$ onto~a~solvable group contains both of~the~associated subgroups. In~particular, $E$~is~not~residually solvable.
\end{eexample}

\enlargethispage{-6pt}

The~next four sections contain proofs of~Theorem and~Corollary.

\section{Certain properties of~the~class~\texorpdfstring{$\mathit{Poly}_{n}\mathcal{C}$}{Poly\_\{n\}C}}\label{esec03}

For~any classes of~groups~$\mathcal{C}$ and~$\mathcal{D}$, let~us~denote by~$\mathcal{C} \cdot \mathcal{D}$ the~class of~groups such that $X \in \mathcal{C} \cdot \mathcal{D}$ if~and~only~if~$X$ is~an~extension of~a~$\mathcal{C}$\nobreakdash-group by~a~$\mathcal{D}$\nobreakdash-group.

\begin{eproposition}\label{ep31}
Let $\mathcal{C}$ be~a~pre-root class of~groups. If~$X$ is~an~extension of~a~group~$Y$ and~the~cardinality of~the~quotient group~$X/Y$ does~not~exceed the~cardinality of~some $\mathcal{C}$\nobreakdash-group~$G$\textup{,} then the~following statements hold.

\textup{1.}\hspace{1ex}Any subgroup $Z \in \mathcal{C}^{*}(Y)$ contains a~subgroup $U \in \mathcal{C}^{*}(Y)$ that is~normal in~$X$.

\textup{2.}\hspace{1ex}If $Y$ is~residually a~$\mathcal{C}$\nobreakdash-group\textup{,} then $X$ is~residually a~$(\mathcal{C} \cdot \{X/Y\})$-group\textup{,} where $\{X/Y\}$ is~the~class consisting of~the~single group~$X/Y$.
\end{eproposition}

\begin{proof}
1.\hspace{1ex}Let $U = \bigcap_{s \in S}s^{-1}Zs$, where $S$ is~a~complete set of~coset representatives for~$Y$ in~$X$. Since $Z$ is~normal in~$Y$, the~equalities $U = \bigcap_{s \in S,\,y \in Y}s^{-1}y^{-1}Zys = \bigcap_{x \in X}x^{-1}Zx$ hold. Hence, $U$ is~normal in~$X$. By~Remak's theorem (see, e.g.,~\cite[Theorem~4.3.9]{KargapolovMerzlyakov1982}), the~quotient group~$Y/U$ can~be~embedded in~the~unrestricted direct product~$P$ of~the~quotient groups~$Y/s^{-1}Zs$ ($s \in S$), each of~which is~isomorphic to~the~$\mathcal{C}$\nobreakdash-group~$Y/Z$. The~cardinality of~$S$ is~equal to~the~cardinality of~$X/Y$ and,~therefore, does~not~exceed the~cardinality of~$G \in \mathcal{C}$. Hence, the~group~$P$ can~be~embedded in~the~unrestricted direct product $Q = \prod_{g \in G}Y_{g}$, where $Y_{g}$ is~an~isomorphic copy of~the~$\mathcal{C}$\nobreakdash-group~$Y/Z$ for~each $g \in G$. Since $\mathcal{C}$ is~a~pre-root class, we~have $Q \in \mathcal{C}$, $P \in \mathcal{C}$, and~$Y/U \in \mathcal{C}$, as~required.

2.\hspace{1ex}Let~us fix an~element $x \in X \setminus \{1\}$ and~find a~homomorphism of~$X$ onto~a~$(\mathcal{C} \cdot \{X/Y\})$-group that takes~$x$ to~a~non-triv\-i\-al element. If~$x \notin Y$, then the~natural homomorphism $X \to X/Y$ is~desired. Since $Y$ is~residually a~$\mathcal{C}$\nobreakdash-group, if~$x \in Y$, we can choose a~subgroup $Z \in \mathcal{C}^{*}(Y)$ that does~not~contain~$x$. By~Statement~1, there exists a~subgroup $U \in \mathcal{C}^{*}(Y)$ that is~normal in~$X$ and~lies in~$Z$. Clearly, $x \notin U$ and~the~natural homomorphism $X \to X/U$ is~the~desired~one.
\end{proof}

Before formulating the~next proposition, let~us recall that every root class~$\mathcal{C}$ satisfies the~\emph{Gruenberg condition}: if~$1 \leqslant Z \leqslant Y \leqslant X$ is~a~subnormal series of~a~group~$X$, $X/Y \in \mathcal{C}$, and~$Y/Z \in \mathcal{C}$, then $Z$ contains a~subgroup of~$\mathcal{C}^{*}(X)$~\cite{Sokolov2015CA}.

\begin{eproposition}\label{ep32}
If $\mathcal{C}$ is~a~pre-root class of~groups\textup{,} then the~following statements hold.{\parfillskip=0pt{}\par}

\textup{1.}\hspace{1ex}If $X_{1} \leqslant \ldots \leqslant X_{n} = X$ \textup{(}$n \geqslant 1$\textup{)} is~a~subnormal sequence of~subgroups of~a~group~$X$ whose factors belong to~$\mathcal{C}$\textup{,} then there exists a~sequence $Y_{1} \leqslant \ldots \leqslant Y_{n} = X$ of~normal subgroups such that $Y_{i+1}/Y_{i} \in \mathcal{C}$ and~$Y_{i} \leqslant X_{i}$ for~all $i \in \{1,\,\ldots,\,n-1\}$. It~follows that $X \in \mathit{Poly}_{n}\mathcal{C}$ \textup{(}$n \geqslant 1$\textup{)} if~and~only~if~$X$ has a~normal series of~length at~most~$n$ whose factors belong to~$\mathcal{C}$.

\textup{2.}\hspace{1ex}Let $1 \leqslant Z \leqslant Y \leqslant X$ be~a~subnormal series of~a~group~$X$. If~$X/Y \in \mathit{Poly}_{m}\mathcal{C}$ and~$Y/Z \in \mathit{Poly}_{n}\mathcal{C}$ for~some $m,n \geqslant 1$\textup{,} then there exists a~subgroup $U \in (\mathit{Poly}_{m+n}\mathcal{C})^{*}(X)$ lying in~$Z$. In~particular\textup{,} $\mathcal{C}$ satisfies the~\emph{weakened Gruenberg condition}\textup{:} if~$X/Y \in \mathcal{C}$ and~$Y/Z \in \mathcal{C}$\textup{,} then $Z$ contains a~subgroup of~$(\mathcal{C} \cdot \mathcal{C})^{*}(X)$.

\textup{3.}\hspace{1ex}Let $X$ be~an~extension of~a~group~$Y$ by~a~$(\mathit{Poly}_{m}\mathcal{C})$-group \textup{(}$m \geqslant 1$\textup{)}. If~$Y$ is~residually a~$(\mathit{Poly}_{n}\mathcal{C})$-group for~some $n \geqslant 1$\textup{,} then $X$ is~residually a~$(\mathit{Poly}_{m+n}\mathcal{C})$-group.
\end{eproposition}

\enlargethispage{-6pt}

\begin{proof}
1.\hspace{1ex}It~is~clear that we can take~$X_{n-1}$ as~$Y_{n-1}$. Suppose that the~subgroup~$Y_{i}$ ($2 \leqslant i \leqslant n-1$) is~already~defined. Since $Y_{i}/Y_{i} \cap X_{i-1} \cong Y_{i}X_{i-1}/X_{i-1} \leqslant X_{i}/X_{i-1} \in \mathcal{C}$ and~$\mathcal{C}$ is~closed under taking subgroups, we have $Y_{i}/Y_{i} \cap X_{i-1} \in \mathcal{C}$. As~noted in~Introduction, the~$(\mathit{Poly}_{n-i}\mathcal{C})$-group~$X/Y_{i}$ is~equinumerous to~a~$\mathcal{C}$\nobreakdash-group. It~follows from~this fact and~Proposition~\ref{ep31} (applied to~the~subgroups~$Y_{i}$ and~$X_{i-1} \cap Y_{i}$) that there exists a~subgroup $Y_{i-1} \in \mathcal{C}^{*}(Y_{i})$ normal in~$X$ and~lying in~$X_{i-1} \cap Y_{i}$. Obviously, this subgroup is~the~desired~one.

\pagebreak

2.\hspace{1ex}Since $X/Y \in \mathit{Poly}_{m}\mathcal{C}$ and~$Y/Z \in \mathit{Poly}_{n}\mathcal{C}$, the~group~$X$ has a~subnormal sequence of~subgroups $Z \leqslant Y_{1} \leqslant \ldots \leqslant Y_{n} = Y \leqslant X_{1} \leqslant \ldots \leqslant X_{m} = X$ with~factors from~$\mathcal{C}$. By~applying Statement~1 to~this sequence, we get the~desired subgroup~$U$.

\enlargethispage{-6pt}

3.\hspace{1ex}As~above, the~$(\mathit{Poly}_{m}\mathcal{C})$-group~$X/Y$ is~equinumerous to~a~$\mathcal{C}$\nobreakdash-group. Since $\mathcal{C} \subseteq \mathit{Poly}_{n}\mathcal{C}$ and~$\mathit{Poly}_{n}\mathcal{C}$ is~a~pre-root class, $X$~is~residually a~$(\mathit{Poly}_{n}\mathcal{C} \cdot \{X/Y\})$-group by~Proposition~\ref{ep31}. It~remains to~note that $\mathit{Poly}_{n}\mathcal{C} \cdot \{X/Y\} \subseteq \mathit{Poly}_{n}\mathcal{C} \cdot \mathit{Poly}_{m}\mathcal{C} \subseteq \mathit{Poly}_{m+n}\mathcal{C}$.
\end{proof}

\section{Some auxiliary statements}\label{esec04}

Throughout the~paper, let~us use the~following denotations:

\noindent
\makebox[7ex][l]{$\ker\sigma$}the~kernel of~a~homomorphism~$\sigma$;

\noindent
\makebox[7ex][l]{$\mathfrak{T}(X)$}the~torsion subset of~a~nilpotent group~$X$;

\noindent
\makebox[7ex][l]{$\zeta_{i}X$}the~$i$\nobreakdash-th~member of~the~upper central series of~a~nilpotent group~$X$.

\begin{eproposition}\label{ep41}
Given a~nilpotent group~$X$\textup{,} the~following statements hold.

\textup{1.}\hspace{1ex}The~set~$\mathfrak{T}(X)$ is~a~normal subgroup of~$X$~\textup{\cite[Theorem~5.2]{ClementMajewiczZyman2017}}.

\textup{2.}\hspace{1ex}If $X$ is~a~tor\-sion-free group\textup{,} then all members of~its~upper central series are~isolated~\textup{\cite[Lemma~5.9]{ClementMajewiczZyman2017}}.
\end{eproposition}

\begin{eproposition}\label{ep42}
If $X$ is~a~tor\-sion-free nilpotent group and~$Y$ is~a~non-triv\-i\-al locally cyclic subgroup of~$X$\textup{,} then $Y \leqslant \zeta_{\lambda(Y)}X$ and~$Y \cap \zeta_{\lambda(Y)-1}X = 1$ for~some number $\lambda(Y) \geqslant 1$.
\end{eproposition}

\begin{proof}
Since $X$ is~nilpotent and~$Y \ne 1$, there exist a~number $\lambda(Y) \geqslant 1$ and~an~element $y \in Y \setminus \{1\}$ such that $y \in \zeta_{\lambda(Y)}X \setminus \zeta_{\lambda(Y)-1}X$. If~$y^{\prime} \in Y \cap \zeta_{\lambda(Y)-1}X$, then~$y$ and~$y^{\prime}$ generate a~cyclic subgroup~$Z$ of~$Y$. The~subgroup~$\zeta_{\lambda(Y)-1}X$ is~isolated in~the~tor\-sion-free group~$X$ by~Proposition~\ref{ep41}. Since $y \in Z \setminus \zeta_{\lambda(Y)-1}X$, it~follows that $Z \cap \zeta_{\lambda(Y)-1}X = 1$ and~$y^{\prime} = 1$. Hence, $Y \cap \zeta_{\lambda(Y)-1}X = 1$. Similarly, if~$y^{\prime\prime} \in Y \cap \zeta_{\lambda(Y)+1}X \setminus \zeta_{\lambda(Y)}X$, then the~equality $\operatorname{sgp}\{y,y^{\prime\prime}\} \cap \zeta_{\lambda(Y)}X = 1$ holds, which contradicts the~choice of~$y$. Thus, $Y \leqslant \zeta_{\lambda(Y)}X$.
\end{proof}

\begin{eproposition}\label{ep43}
\textup{\cite[Proposition~3.1]{Sokolov2024JA}}
Let $\mathcal{C}$ be~a~class of~groups closed under taking subgroups and~finite direct products. If~$X$ is~a~group with~subgroups~$Y$ and~$Z$\textup{,} then the~following statements hold.

\textup{1.}\hspace{1ex}If $Y \in \mathcal{C}^{*}(X)$\textup{,} then $Y \cap Z \in \mathcal{C}^{*}(Z)$. As~a~result\textup{,} any subgroup of~a~residually $\mathcal{C}$\nobreakdash-group is~itself residually a~$\mathcal{C}$\nobreakdash-group.

\textup{2.}\hspace{1ex}If $Y,Z \in \mathcal{C}^{*}(X)$\textup{,} then $Y \cap Z \in \mathcal{C}^{*}(X)$. As~a~result\textup{,} any finite subset of~non-triv\-i\-al elements of~a~residually $\mathcal{C}$\nobreakdash-group does~not~intersect some subgroup of~$\mathcal{C}^{*}(X)$.
\end{eproposition}

Recall (see, e.g.,~\cite[Section~2.2]{DixonKurdachenkoSubbotin2017} that a~group~$X$ is~said to~have a~\emph{finite $0$\nobreakdash-rank} if~there exists a~finite subnormal series of~$X$ each of~whose factor is~either a~periodic or~an~infinite cyclic group.

\begin{eproposition}\label{ep44}
\textup{\cite[Proposition~18]{Tumanova2015IVM}}
Let $\mathcal{C}$ be~a~class of~tor\-sion-free groups closed under taking subgroups and~finite direct products. If~$X$ is~residually a~$\mathcal{C}$\nobreakdash-group and~$Y$ is~a~subgroup of~$X$ having a~finite $0$\nobreakdash-rank\textup{,} then there exists a~subgroup $Z \in \mathcal{C}^{*}(X)$ such that $Z \cap Y = 1$.
\end{eproposition}

The~next proposition follows from~the~results of~\cite{Cohen1974JAMS}.

\begin{eproposition}\label{ep45}
Let $\mathcal{T}$ contain only one edge~$e$. If~$N$ is~a~normal subgroup of~$\pi_{1}(\mathcal{G}(\mathcal{T}))$ that intersects trivially with~$H_{+e}$ and~$H_{-e}$\textup{,} then it splits as~the~free product of~a~free subgroup\textup{,} the~subgroups~$N \cap G_{e(1)}$ and~$N \cap G_{e(-1)}$\textup{,} and~also some subgroups of~the~form $x^{-1}(N \cap G_{e(\varepsilon)})x$\textup{,} where $x \in \pi_{1}(\mathcal{G}(\mathcal{T})) \setminus \{1\}$ and~$\varepsilon = \pm 1$.
\end{eproposition}

Using Proposition~\ref{ep45} and~induction on~the~number of~vertices, we can deduce

\begin{eproposition}\label{ep46}
Let $\mathcal{T}$ be~finite. If~$N$ is~a~normal subgroup of~$\pi_{1}(\mathcal{G}(\mathcal{T}))$ that intersects trivially with~each subgroup~$H_{\varepsilon e}$ \textup{(}$e \in \mathcal{E}$\textup{,} $\varepsilon = \pm 1$\textup{),} then it splits as~the~free product of~a~free subgroup\textup{,} the~subgroups~$N \cap G_{v}$ \textup{(}$v \in \mathcal{V}$\textup{),} and~also some subgroups of~the~form $x^{-1}(N \cap G_{v})x$ \textup{(}$v \in \mathcal{V}$\textup{,} $x \in \pi_{1}(\mathcal{G}(\mathcal{T})) \setminus \{1\}$\textup{)}.
\end{eproposition}

It~is~easy to~see that if~$H_{\varepsilon e}$ lies in~the~center of~$G_{e(\varepsilon)}$ for~any $e \in \mathcal{E}$ and~$\varepsilon = \pm 1$, then $S = \operatorname{sgp}\{h\varphi_{+e}(h\varphi_{-e})^{-1} \mid e \in \mathcal{E},\ h \in H_{e}\}$ is~a~normal subgroup of~the~direct product~$D$ of~the~groups~$G_{v}$ ($v \in \mathcal{V}$). Following~\cite{SokolovTumanova2019AL}, we will~call the~quotient group $\mathfrak{D}(\mathcal{G}(\mathcal{T})) = D/S$ the~\emph{generalized direct product associated with~the~graph of~groups~$\mathcal{G}(\mathcal{T})$}.

\begin{eproposition}\label{ep47}
\textup{\cite[Theorem~1]{SokolovTumanova2019AL}}
If $H_{\varepsilon e}$ lies in~the~center of~$G_{e(\varepsilon)}$ for~any $e \in \mathcal{E}$ and~$\varepsilon = \pm 1$\textup{,} then each group~$G_{v}$ \textup{(}$v \in \mathcal{V}$\textup{)} can~be~embedded in~$\mathfrak{D}(\mathcal{G}(\mathcal{T}))$ by~the~identity mapping of~its~generators.
\end{eproposition}

Let $\omega$ be~the~mapping of~words that continues the~identity mapping of~the~generators of~$\pi_{1}(\mathcal{G}(\mathcal{T}))$ into~$\mathfrak{D}(\mathcal{G}(\mathcal{T}))$. Clearly, $\omega$~takes all defining relations of~$\pi_{1}(\mathcal{G}(\mathcal{T}))$ to~the~equalities valid in~$\mathfrak{D}(\mathcal{G}(\mathcal{T}))$ and,~therefore, defines a~homomorphism~$\sigma$. For~each $v \in \mathcal{V}$, consider the~diagram
\[
\begin{CD}
G_{v} @>{\iota}>> \pi_{1}(\mathcal{G}(\mathcal{T})) \\
@| @VV{\sigma}V \\ 
G_{v} @>{\iota}>> \mathfrak{D}(\mathcal{G}(\mathcal{T})),
\end{CD}
\]
where $\iota$ denotes the~embeddings by~the~identity mappings of~the~generators. Since every such diagram is~commutative, $\sigma$~acts injectively on~all subgroups~$G_{v}$ ($v \in \mathcal{V}$).

\begin{eproposition}\label{ep48}
Let $\mathcal{N}$ be~a~pre-root class of~nilpotent groups closed under taking quotient groups\textup{,} and~let $\mathcal{N}^{\mathfrak{tf}}$ be~the~class of~all tor\-sion-free $\mathcal{N}$\nobreakdash-groups. If~$\mathcal{T}$ is~finite\textup{,} $G_{v} \in \mathcal{N}^{\mathfrak{tf}}$ for~each $v \in \mathcal{V}$\textup{,} and~$H_{\varepsilon e}$ lies in~the~center of~$G_{e(\varepsilon)}$ for~any $e \in \mathcal{E}$ and~$\varepsilon = \pm 1$\textup{,} then there exists a~homomorphism of~$\pi_{1}(\mathcal{G}(\mathcal{T}))$ onto~an~$\mathcal{N}^{\mathfrak{tf}}$\nobreakdash-group that acts injectively on~all subgroups~$G_{v}$ \textup{(}$v \in \mathcal{V}$\textup{)}.
\end{eproposition}

\begin{proof}
Suppose that $\mathfrak{D}(\mathcal{G}(\mathcal{T}))$ is~the~generalized direct product associated with~$\mathcal{G}(\mathcal{T})$ and~$\sigma\colon \pi_{1}(\mathcal{G}(\mathcal{T})) \to \mathfrak{D}(\mathcal{G}(\mathcal{T}))$ is~the~homomorphism that continues the~identity mapping of~the~generators of~$\pi_{1}(\mathcal{G}(\mathcal{T}))$. Since the~class~$\mathcal{N}$ is~closed under taking finite direct products and~quotient groups, it~contains~$\mathfrak{D}(\mathcal{G}(\mathcal{T}))$. The~natural homomorphism $\eta\colon \mathfrak{D}(\mathcal{G}(\mathcal{T})) \to \mathfrak{D}(\mathcal{G}(\mathcal{T}))/\mathfrak{T}(\mathfrak{D}(\mathcal{G}(\mathcal{T})))$ acts injectively on~all subgroups~$G_{v}$ ($v \in \mathcal{V}$) because they are~tor\-sion-free. It~remains to~note that $\sigma$ has the~same property and,~therefore, the~composition~$\sigma\eta$ is~the~desired mapping.
\end{proof}

\section{Proof of~Statement~\texorpdfstring{$(\alpha)$}{(alpha)}}\label{esec05}

In~what follows, given~$v \in \mathcal{V}$, let $\Theta(v) = \{(e,\varepsilon) \mid e \in \mathcal{E},\ \varepsilon = \pm 1,\ v = e(\varepsilon)\}$.

\begin{eproposition}\label{ep51}
Suppose that $\mathcal{C}$ is~a~pre-root class of~groups\textup{,} $\mathcal{T}$~is~a~finite tree\textup{,} and~for each $v \in \mathcal{V}$\textup{,} a~number $n_{v} \geqslant 1$ is~defined such that $G_{v} \in \mathit{Poly}_{n_{v}}\mathcal{C}$. Suppose also that\textup{,} for~some $n \geqslant 1$\textup{,} there exists a~homomorphism~$\sigma$ of~$\pi_{1}(\mathcal{G}(\mathcal{T}))$ onto~a~$(\mathit{Poly}_{n}\mathcal{C})$-group acting injectively on~all subgroups~$H_{\varepsilon e}$ \textup{(}$e \in \mathcal{E}$\textup{,} $\varepsilon = \pm 1$\textup{)}. If~$u \in \mathcal{V}$ and~$K \leqslant G_{u}$\textup{,} then one can find a~homomorphism of~$\pi_{1}(\mathcal{G}(\mathcal{T}))$ onto~a~$(\mathit{Poly}_{n+n_{u}}\mathcal{C})$-group that acts injectively on~$K$ and~all subgroups~$H_{\varepsilon e}$ \textup{(}$e \in \mathcal{E}$\textup{,} $\varepsilon = \pm 1$\textup{)}.
\end{eproposition}

\begin{proof}
Let $N = \ker\sigma$. By~Proposition~\ref{ep46}, the~subgroup $N \cap G_{u}$ is~a~free factor of~$N$. Hence, there exists a~homomorphism $\tau\colon N \to N \cap G_{u}$ acting identically on~$N \cap G_{u}$. Let~us consider the~subnormal series $1 \leqslant M \leqslant N \leqslant \pi_{1}(\mathcal{G}(\mathcal{T}))$, \pagebreak where $M = \ker\tau$, and~note that $\pi_{1}(\mathcal{G}(\mathcal{T}))/N \in \mathit{Poly}_{n}\mathcal{C}$. Since $G_{u} \in \mathit{Poly}_{n_{u}}\mathcal{C}$ and~the~class~$\mathit{Poly}_{n_{u}}\mathcal{C}$ is~closed under taking subgroups, $N/M \cong N \cap G_{u} \in \mathit{Poly}_{n_{u}}\mathcal{C}$. Hence, by~Proposition~\ref{ep32}, there exists a~subgroup $L \in (\mathit{Poly}_{n+n_{u}}\mathcal{C})^{*}(\pi_{1}(\mathcal{G}(\mathcal{T})))$ lying in~$M$. Let $k \in K \setminus \{1\}$. If~$k \notin N$, then $k \notin L$ because $L \leqslant M \leqslant N$. If~$k \in N$, then $k \in N \cap G_{u}$ and, since $\tau$ acts identically on~$N \cap G_{u}$, the~relations $k \notin M$ and~$k \notin L$ hold. Thus, $K \cap L = 1$. It~also follows from~the~definition of~$\sigma$ that $H_{\varepsilon e} \cap L = 1$ for~any $e \in \mathcal{E}$ and~$\varepsilon = \pm 1$. Therefore, the~natural homomorphism $\pi_{1}(\mathcal{G}(\mathcal{T})) \to \pi_{1}(\mathcal{G}(\mathcal{T}))/L$ is~the~desired~one.
\end{proof}

For~each $v \in \mathcal{V}$, let $R_{v}$ be~a~normal subgroup of~$G_{v}$.~If
\begin{equation}\label{eq}
(R_{e(1)} \cap H_{+e})(\varphi_{+e})^{-1} = 
(R_{e(-1)} \cap H_{-e})(\varphi_{-e})^{-1}
\end{equation}
for~any $e \in \mathcal{E}$, we~will~call the~family $\mathcal{R} = \{R_{v} \mid v \in \mathcal{V}\}$ a~\emph{system of~compatible normal subgroups in~the~graph of~groups~$\mathcal{G}(\mathcal{T})$}. Let~us~put $\overline{G}_{v} = G_{v}/R_{v}$ ($v \in \mathcal{V}$)~and
\[
\overline{H}_{e} = 
H_{e}/(R_{e(1)} \cap H_{+e})(\varphi_{+e})^{-1} = 
H_{e}/(R_{e(-1)} \cap H_{-e})(\varphi_{-e})^{-1}
\]
($e \in \mathcal{E}$, $\varepsilon = \pm 1$). Let also $\overline{\varphi}_{\varepsilon e}\colon \overline{H}_{e} \to \overline{G}_{e(\varepsilon)}$ ($e \in \mathcal{E}$, $\varepsilon = \pm 1$) be~the~mapping that takes the~coset~$\bar h$ ($h \in H_{e}$) to~the~element~$(h\varphi_{\varepsilon e})R_{e(\varepsilon)}$. It~follows from~the~equalities~\eqref{eq} ($e \in \mathcal{E}$) that $\overline{\varphi}_{\varepsilon e}$ is~well defined and~is~an~isomorphism of~$\overline{H}_{e}$ onto~$H_{\varepsilon e}R_{e(\varepsilon)}/R_{e(\varepsilon)}$. Therefore, we can consider the~graph of~groups~$\mathcal{G}_{\mathcal{R}}(\mathcal{T})$ whose vertices and~edges are~associated with~the~groups~$\overline{G}_{v}$ ($v \in \mathcal{V}$), the~groups~$\overline{H}_{e}$ ($e \in \mathcal{E}$), and~the~isomorphisms~$\overline{\varphi}_{\pm e}$ ($e \in \mathcal{E}$). It~is~easy to~see that the~identity mapping of~the~generators of~$\pi_{1}(\mathcal{G}(\mathcal{T}))$ into~$\pi_{1}(\mathcal{G}_{\mathcal{R}}(\mathcal{T}))$ defines a~surjective homomorphism $\rho_{\mathcal{R}}\colon \pi_{1}(\mathcal{G}(\mathcal{T})) \to \pi_{1}(\mathcal{G}_{\mathcal{R}}(\mathcal{T}))$, whose kernel coincides with~the~normal closure of~the~set $\bigcup_{v \in \mathcal{V}}R_{v}$ in~$\pi_{1}(\mathcal{G}(\mathcal{T}))$. It~also follows from~the~definition of~$\mathcal{G}_{\mathcal{R}}(\mathcal{T})$ that $\rho_{\mathcal{R}}$ maps the~vertex groups and~edge subgroups of~$\pi_{1}(\mathcal{G}(\mathcal{T}))$ onto~the~corresponding vertex groups and~edge subgroups of~$\pi_{1}(\mathcal{G}_{\mathcal{R}}(\mathcal{T}))$, and~continues the~natural homomorphisms $G_{v} \to G_{v}/R_{v}$ ($v \in \mathcal{V}$).

\begin{eproposition}\label{ep52}
Suppose that $\mathcal{N}$ is~a~pre-root class of~nilpotent groups closed under taking quotient groups and~$\mathcal{N}^{\mathfrak{tf}}$ is~the~class of~all tor\-sion-free $\mathcal{N}$\nobreakdash-groups. Suppose also that $\mathcal{T}$ is~a~finite tree with~$n \geqslant 1$ vertices and~each subgroup $H_{\varepsilon e}$ \textup{(}$e \in \mathcal{E}$\textup{,} $\varepsilon = \pm 1$\textup{)} \mbox{is~a~non-triv\-i\-al} tor\-sion-free\kern-1pt{} locally\kern-1pt{} cyclic\kern-1pt{} group.\kern-1pt{} If\kern-1pt{}~$G_{v}\kern-2pt{} \in\kern-1pt{} \mathcal{N}^{\mathfrak{tf}}$\kern-1pt{} for\kern-1pt{}~all\kern-1pt{} $v\kern-1.5pt{} \in\kern-1.5pt{} \mathcal{V}$\kern-1pt{}\textup{,}\kern-1.5pt{} then\kern-1pt{}~the\kern-1pt{}~\mbox{following}\kern-1pt{}~\mbox{statements}\kern-1pt{}~hold.

\textup{1.}\hspace{1ex}There exists a~homomorphism of~$\pi_{1}(\mathcal{G}(\mathcal{T}))$ onto~a~$(\mathit{Poly}_{\max\{1,n-1\}}\mathcal{N}^{\mathfrak{tf}})$-group that acts injectively on~each subgroup~$H_{\varepsilon e}$ \textup{(}$e \in \mathcal{E}$\textup{,} $\varepsilon = \pm 1$\textup{)}.

\textup{2.}\hspace{1ex}For~any vertex $u \in \mathcal{V}$ and~for~any subgroup $K \leqslant G_{u}$\textup{,} there exists a~homomorphism of~$\pi_{1}(\mathcal{G}(\mathcal{T}))$ onto~a~$(\mathit{Poly}_{n}\mathcal{N}^{\mathfrak{tf}})$-group that acts injectively on~$K$ and~each subgroup~$H_{\varepsilon e}$ \textup{(}$e \in \mathcal{E}$\textup{,} $\varepsilon = \pm 1$\textup{)}.
\end{eproposition}

\begin{proof}
Let~us~prove both statements simultaneously using induction on~$n$. If~$n = 1$, the~identity mapping of~the~unique vertex group is~the~desired homomorphism. Therefore, we can assume in~what follows that $n > 1$. By~Proposition~\ref{ep42}, for~any $e \in \mathcal{E}$ and~$\varepsilon = \pm 1$, there exists a~number $\lambda_{\varepsilon e} \geqslant 1$ such that $H_{\varepsilon e} \leqslant \zeta_{\lambda_{\varepsilon e}}G_{e(\varepsilon)}$ and~$H_{\varepsilon e} \cap \zeta_{\lambda_{\varepsilon e}-1}G_{e(\varepsilon)} = 1$. Given $v \in \mathcal{V}$, let~us~put $\lambda_{v} = \min\{\lambda_{\varepsilon e} \mid (e,\varepsilon) \in \Theta(v)\} - 1$ and~$\mu_{\varepsilon e} = \lambda_{\varepsilon e} - \lambda_{v}$ ($(e,\varepsilon) \in \Theta(v)$). Since the~subgroup~$\zeta_{\lambda_{v}}G_{v}$ is~isolated in~$G_{v}$ by~Proposition~\ref{ep41} and~the~class~$\mathcal{N}$ is~closed under taking quotient groups, we have $G_{v}/\zeta_{\lambda_{v}}G_{v} \in \mathcal{N}^{\mathfrak{tf}}$. It~follows from~the~definition of~$\lambda_{v}$ that the~natural homomorphism $\eta_{v}\colon G_{v} \to G_{v}/\zeta_{\lambda_{v}}G_{v}$ acts injectively on~all subgroups~$H_{\varepsilon e}$ ($(e,\varepsilon) \in \Theta(v)$) and, for~at~least one pair $(e,\varepsilon) \in \Theta(v)$, the~equality $\mu_{\varepsilon e} = 1$ holds. The~latter means that the~subgroup~$H_{\varepsilon e}\eta_{v}$ lies in~the~center of~$G_{v}\eta_{v}$. Let~us~prove the~existence of~an~edge $f \in \mathcal{E}$ with~the~property $\mu_{+f} = 1 = \mu_{-f}$.

\enlargethispage{-6pt}

First of~all, we transform~$\mathcal{T}$ into~a~directed graph~$\overset{\kern.4ex{}\rightarrow}{\mathcal{T}}$ as~follows. If~$e \in \mathcal{E}$ and~$\mu_{\varepsilon e} < \mu_{-\varepsilon e}$ for~some $\varepsilon = \pm 1$, \pagebreak let~$e(\varepsilon)$ and~$e(-\varepsilon)$ be~the~origin and~target vertices of~$e$, respectively. If~$\mu_{\varepsilon e} = \mu_{-\varepsilon e}$, we orient the~edge~$e$ arbitrarily. Clearly, the~graph~$\overset{\kern.4ex{}\rightarrow}{\mathcal{T}}$ is~finite and~does~not~contain closed paths. Therefore, starting from~some vertex $v \in \mathcal{V}$ and~moving arbitrarily along the~arcs of~$\overset{\kern.4ex{}\rightarrow}{\mathcal{T}}$, we will inevitably arrive at~a~vertex $v^{\prime} \in \mathcal{V}$ that is~not~an~origin for~any arc. But,~as~noted above, the~equality $\mu_{\varepsilon f} = 1$ holds for~at~least one pair $(f,\varepsilon) \in \Theta(v^{\prime})$. Hence, $\mu_{-\varepsilon f} = \mu_{\varepsilon f} = 1$ and~$f$ is~the~desired edge.

Since each subgroup $R_{v} = \ker\eta_{v}$ ($v \in \mathcal{V}$) intersects trivially with~all subgroups~$H_{\varepsilon e}$ ($(e,\varepsilon) \in \Theta(v)$), the~family $\mathcal{R} = \{R_{v} \mid v \in \mathcal{V}\}$ is~a~system of~compatible normal subgroups. Therefore, the~graph of~groups~$\mathcal{G}_{\mathcal{R}}(\mathcal{T})$ and~the~homomorphism $\rho_{\mathcal{R}}\colon \pi_{1}(\mathcal{G}(\mathcal{T})) \to \pi_{1}(\mathcal{G}_{\mathcal{R}}(\mathcal{T}))$ are~defined. The~map~$\rho_{\mathcal{R}}$ extends the~homomorphisms~$\eta_{v}$ ($v \in \mathcal{V}$) and,~therefore, acts injectively on~all subgroups~$H_{\varepsilon e}$ ($e \in \mathcal{E}$, $\varepsilon = \pm 1$). Let~$\mathcal{T}_{f}$ be~the~subtree of~$\mathcal{T}$ consisting of~the~edge~$f$ and~the~vertices~$f(1)$ and~$f(-1)$. It~is~clear that the~group~$\pi_{1}(\mathcal{G}_{\mathcal{R}}(\mathcal{T}))$ can~be~considered as~the~fundamental group of~the~graph of~groups~$\mathcal{G}_{\mathcal{R}}^{\prime}(\mathcal{T}^{\prime})$ obtained from~$\mathcal{G}_{\mathcal{R}}(\mathcal{T})$ by~removing the~edge~$f$, merging its~ends, and~assigning the~group~$\pi_{1}(\mathcal{G}_{\mathcal{R}}(\mathcal{T}_{f}))$ to~this new vertex~$w$. Let~us~put $\mathcal{V}^{\prime} = \mathcal{V} \setminus \{f(1),f(-1)\} \cup \{w\}$.

It~follows from~the~choice of~$f$ that the~subgroups~$H_{+f}\rho_{\mathcal{R}}$ and~$H_{-f}\rho_{\mathcal{R}}$ lie in~the~centers of~the~groups~$G_{f(1)}\rho_{\mathcal{R}}$ and~$G_{f(-1)}\rho_{\mathcal{R}}$, respectively. Therefore, by~Proposition~\ref{ep48}, $\pi_{1}(\mathcal{G}_{\mathcal{R}}(\mathcal{T}_{f}))$ has a~homomorphism~$\psi$ onto~an~$\mathcal{N}^{\mathfrak{tf}}$\nobreakdash-group that acts injectively on~$G_{f(1)}\rho_{\mathcal{R}}$ and~$G_{f(-1)}\rho_{\mathcal{R}}$. In~particular, $\psi$~is~injective on~all edge subgroups of~the~graphs~$\mathcal{G}_{\mathcal{R}}(\mathcal{T})$ and~$\mathcal{G}_{\mathcal{R}}^{\prime}(\mathcal{T}^{\prime})$ contained in~$G_{f(1)}\rho_{\mathcal{R}}$ and~$G_{f(-1)}\rho_{\mathcal{R}}$. If~$S_{w} = \ker\psi$ and~$S_{v} = 1$ for~all $v \in \mathcal{V}^{\prime} \setminus \{w\}$, then the~family $\mathcal{S} = \{S_{v} \mid v \in \mathcal{V}^{\prime}\}$ is~a~system of~compatible normal subgroups in~$\mathcal{G}_{\mathcal{R}}^{\prime}(\mathcal{T}^{\prime})$. Hence, one can define the~graph of~groups~$(\mathcal{G}_{\mathcal{R}}^{\prime})_{\mathcal{S}}^{\vphantom{\prime}}(\mathcal{T}^{\prime})$ and~the~homomorphism $\rho^{\prime}_{\mathcal{S}}\colon \pi_{1}(\mathcal{G}_{\mathcal{R}}^{\prime}(\mathcal{T}^{\prime})) \to \pi_{1}((\mathcal{G}_{\mathcal{R}}^{\prime})_{\mathcal{S}}^{\vphantom{\prime}}(\mathcal{T}^{\prime}))$, which extends~$\psi$ and~the~identity mappings of~$G_{v}\rho_{\mathcal{R}}$ ($v \in \mathcal{V}^{\prime} \setminus \{w\}$). The~latter means that $\rho^{\prime}_{\mathcal{S}}$ is~injective on~all subgroups~$H_{\varepsilon e}\rho_{\mathcal{R}}$ ($e \in \mathcal{E}$, $\varepsilon = \pm 1$).

Since $\pi_{1}((\mathcal{G}_{\mathcal{R}}^{\prime})_{\mathcal{S}}^{\vphantom{\prime}}(\mathcal{T}^{\prime}))$ satisfies the~conditions of~this proposition and~$\mathcal{T}^{\prime}$ has $n-1$~vertices, the~inductive hypothesis can~be~applied to~the~group~$\pi_{1}((\mathcal{G}_{\mathcal{R}}^{\prime})_{\mathcal{S}}^{\vphantom{\prime}}(\mathcal{T}^{\prime}))$, the~vertex~$w$, and~the~subgroup $K^{\prime} = (H_{+f}\rho_{\mathcal{R}}^{\vphantom{\prime}})\rho^{\prime}_{\mathcal{S}} = (H_{-f}\rho_{\mathcal{R}}^{\vphantom{\prime}})\rho^{\prime}_{\mathcal{S}}$. By~this hypothesis, $\pi_{1}((\mathcal{G}_{\mathcal{R}}^{\prime})_{\mathcal{S}}^{\vphantom{\prime}}(\mathcal{T}^{\prime}))$ has a~homomorphism~$\sigma$ onto~a~$(\mathit{Poly}_{n-1}\mathcal{N}^{\mathfrak{tf}})$-group that is~injective on~the~subgroups~$K^{\prime}$ and~$H_{\varepsilon e}\rho_{\mathcal{R}}^{\vphantom{\prime}}\rho^{\prime}_{\mathcal{S}}$ ($e \in \mathcal{E} \setminus \{f\}$, $\varepsilon = \pm 1$). Therefore, the~composition~$\rho_{\mathcal{R}}^{\vphantom{\prime}}\rho^{\prime}_{\mathcal{S}}\sigma$ acts injectively on~all subgroups~$H_{\varepsilon e}$ ($e \in \mathcal{E}$, $\varepsilon = \pm 1$), and~Statement~1 is~proved. Statement~2 follows from~Statement~1 and~Proposition~\ref{ep51}.
\end{proof}

\begin{proof}[\textbf{\textup{Proof of~Statement~$\boldsymbol{(\alpha)}$}}]
For~each $v \in \mathcal{V}$, let~us~denote by~$R_{v}$ the~kernel of~the~natural homomorphism $\eta_{v}\colon G_{v} \to G_{v}/\mathfrak{T}(G_{v})$. Since the~subgroups~$H_{\varepsilon e}$ ($e \in \mathcal{E}$, $\varepsilon = \pm 1$) are~tor\-sion-free, the~equality $H_{\varepsilon e} \cap R_{e(\varepsilon)} = 1$ holds for~any $e \in \mathcal{E}$ and~$\varepsilon = \pm 1$. Hence, the~family $\mathcal{R} = \{R_{v} \mid v \in \mathcal{V}\}$ is~a~system of~compatible normal subgroups, and~we can consider the~graph of~groups~$\mathcal{G}_{\mathcal{R}}(\mathcal{T})$ and~the~homomorphism $\rho_{\mathcal{R}}\colon \pi_{1}(\mathcal{G}(\mathcal{T})) \to \pi_{1}(\mathcal{G}_{\mathcal{R}}(\mathcal{T}))$. The~latter extends the~mappings~$\eta_{v}$ ($v \in \mathcal{V}$) and,~therefore, acts injectively on~the~subgroups~$H_{\varepsilon e}$ ($e \in \mathcal{E}$, $\varepsilon = \pm 1$). Since $G_{v}\rho_{\mathcal{R}} = G_{v}/\mathfrak{T}(G_{v}) \in \mathcal{N}^{\mathfrak{tf}}$ for~all $v \in \mathcal{V}$, the~homomorphism~$\rho_{\mathcal{R}}$ can~be~extended to~the~desired one due~to~Proposition~\ref{ep52}.
\end{proof}

\section{Proof of~Statement~\texorpdfstring{$(\beta)$}{(beta)}, Statement~\texorpdfstring{$(\gamma)$}{(gamma)}, and~Corollary}\label{esec06}

Let $P$ be~the~free product of~groups~$X_{i}$ ($i \in \mathcal{I}$). Recall that the~\emph{normal form} of~an~element $x \in P$ is~a~presentation of~this element as~a~product $x_{1}x_{2} \ldots x_{l}$, where $l \geqslant 0$, $x_{1},\,x_{2},\,\ldots,\,x_{l} \in \bigcup_{i \in \mathcal{I}}X_{i} \setminus \{1\}$, and~if~$l \geqslant 2$, then, for~any $k \in \{1,\,2,\,\ldots,\,l-1\}$, the~elements~$x_{k}$ and~$x_{k+1}$ lie in~different free factors~$X_{i}$ ($i \in \mathcal{I}$) (it~is~assumed that if~$l = 0$, then $x_{1}x_{2} \ldots x_{l} = 1$). Using Theorem~1.2 from~\cite[Chapter~IV]{LyndonSchupp1980} and~induction on~the~number of~free factors, it~is~easy to~deduce that each element $x \in P$ has a~unique normal form. It~follows that $x \ne 1$ if~and~only~if~the~length~$l$ of~this form is~non\nobreakdash-zero.

\pagebreak

\begin{eproposition}\label{ep61}
If $\mathcal{C}$ is~a~pre-root class of~groups\textup{,} $\Phi$~is~the~class of~all free groups\textup{,} and~$n \geqslant 1$\textup{,} then the~free product~$P$ of~residually $(\mathit{Poly}_{n}\mathcal{C})$-groups~$X_{i}$ \textup{(}$i \in \mathcal{I}$\textup{)} and~a~free group~$Y$ is~residually a~$(\Phi \cdot \mathit{Poly}_{n}\mathcal{C})$-group.
\end{eproposition}

\begin{proof}
If $x \in P \setminus \{1\}$ is~an~arbitrary element and~$x_{1}x_{2} \ldots x_{l}$ is~its normal form, then $l > 0$ and~$x_{1},\,x_{2},\,\ldots,\,x_{l} \in \bigcup_{j \in \mathcal{J}}X_{j} \cup Y \setminus \{1\}$ for~some finite subset~$\mathcal{J}$ of~$\mathcal{I}$. Since~$X_{j}$ ($j \in \mathcal{J}$) are~residually $(\mathit{Poly}_{n}\mathcal{C})$-groups and~the~class~$\mathit{Poly}_{n}\mathcal{C}$ is~closed under taking subgroups and~finite direct products, Proposition~\ref{ep43} guarantees that, for~each $j \in \mathcal{J}$, there exists a~subgroup $R_{j} \in (\mathit{Poly}_{n}\mathcal{C})^{*}(X_{j})$ satisfying the~relation $R_{j} \cap (\{x_{k} \mid 1 \leqslant k \leqslant l\} \cap X_{j}) = \varnothing$. Let~us~consider the~homomorphism~$\sigma$ of~$P$ onto~the~free product of~$Y$ and~the~groups~$X_{j}/R_{j}$ ($j \in \mathcal{J}$) that maps the~free factors~$X_{i}$ ($i \in \mathcal{I} \setminus \mathcal{J}$) to~the~trivial element and~continues the~identity mappings of~the~generators of~$Y$ and~$X_{j}$ ($j \in \mathcal{J}$).

By~the~choice of~$R_{j}$ ($j \in \mathcal{J}$), the~product $x_{1}\sigma x_{2}\sigma \ldots x_{l}\sigma$ is~the~normal form of~$x\sigma$ of~length $l > 0$. Hence, $x\sigma \ne 1$. Let~$\tau$ be~the~homomorphism of~$P\sigma$ onto~the~direct product~$D$ of~the~groups~$X_{j}/R_{j}$ ($j \in \mathcal{J}$) that continues the~identity mappings of~the~factors~$X_{j}/R_{j}$ ($j \in \mathcal{J}$) and~takes~$Y$ to~$1$. It~follows from~the~equalities $\ker\tau \cap z^{-1}(X_{j}/R_{j})z = z^{-1}(\ker\tau \cap (X_{j}/R_{j}))z = 1$ ($j \in \mathcal{J}$, $z \in P\sigma$) and~Kurosh subgroup theorem~\cite{Kurosch1934MA} that $\ker\tau$ splits as~the~free product of~a~free subgroup and~subgroups of~the~form $\ker\tau \cap z^{-1}Yz$ ($z \in P\sigma$). Since $Y$ is~free and~$D \in \mathit{Poly}_{n}\mathcal{C}$, $\ker\tau$~is~also free and~$P\sigma \in \Phi \cdot \mathit{Poly}_{n}\mathcal{C}$. Thus, $P$~is~residually a~$(\Phi \cdot \mathit{Poly}_{n}\mathcal{C})$-group.
\end{proof}

\begin{eproposition}\label{ep62}
Suppose that $\mathcal{C}$ is~a~pre-root class of~groups\textup{,} $\mathcal{T}$~is~finite\textup{,} and~for~some $n \geqslant 1$\textup{,} there exists a~homomorphism~$\sigma$ of~$\pi_{1}(\mathcal{G}(\mathcal{T}))$ onto~a~$(\mathit{Poly}_{n}\mathcal{C})$-group that acts injectively on~all subgroups~$H_{\varepsilon e}$ \textup{(}$e \in \mathcal{E}$\textup{,} $\varepsilon = \pm 1$\textup{)}. If~each~$G_{v}$ \textup{(}$v \in \mathcal{V}$\textup{)} is~residually a~$(\mathit{Poly}_{m}\mathcal{C})$-group and~every free group is~residually a~$(\mathit{Poly}_{l}\mathcal{C})$-group for~some $l,m \geqslant 1$\textup{,} then $\pi_{1}(\mathcal{G}(\mathcal{T}))$ is~residually a~$(\mathit{Poly}_{l+m+n}\mathcal{C})$-group.
\end{eproposition}

\begin{proof}
By~Proposition~\ref{ep46}, the~subgroup $N = \ker\sigma$ splits as~the~free product of~a~free subgroup and~subgroups of~the~form $x^{-1}(N \cap G_{v})x$, where $v \in \mathcal{V}$ and~$x \in \pi_{1}(\mathcal{G}(\mathcal{T}))$. Since~$G_{v}$ ($v \in \mathcal{V}$) are~residually $(\mathit{Poly}_{m}\mathcal{C})$-groups and~the~class~$\mathit{Poly}_{m}\mathcal{C}$ is~closed under taking subgroups and~finite direct products, the~subgroups~$N \cap G_{v}$ ($v \in \mathcal{V}$) are~also residually $(\mathit{Poly}_{m}\mathcal{C})$-groups due~to~Proposition~\ref{ep43}. It~follows from~this fact and~Proposition~\ref{ep61} that $N$ is~residually a~$(\Phi \cdot \mathit{Poly}_{m}\mathcal{C})$-group, where $\Phi$ is~the~class of~all free groups. By~Proposition~\ref{ep32}, every $(\Phi \cdot \mathit{Poly}_{m}\mathcal{C})$-group is~residually a~$(\mathit{Poly}_{l+m}\mathcal{C})$-group. Hence, $N$ is~also residually a~$(\mathit{Poly}_{l+m}\mathcal{C})$-group, while $\pi_{1}(\mathcal{G}(\mathcal{T}))$ is~residually a~$(\mathit{Poly}_{l+m+n}\mathcal{C})$-group due~to~the~same Proposition~\ref{ep32}.
\end{proof}

\begin{proof}[\textbf{\textup{Proof of~Statements~$\boldsymbol{(\beta)}$ and~$\boldsymbol{(\gamma)}$}}]
By~Proposition~\ref{ep44}, if~$G_{e(\varepsilon)}$ ($e \in \mathcal{E}$, $\varepsilon = \pm 1$) is~residually an~$\mathcal{N}^{\mathfrak{tf}}$\nobreakdash-group, then it has a~homomorphism onto~an~$\mathcal{N}^{\mathfrak{tf}}$\nobreakdash-group that is~injective on~$H_{\varepsilon e}$. Therefore, we can prove Statements~$(\beta)$ and~$(\gamma)$ simultaneously assuming further that the~symbol~$\mathcal{C}$ denotes the~class~$\mathcal{N}$ (in~the~case of~Statement~$(\beta)$) or~the~class~$\mathcal{N}^{\mathfrak{tf}}$ (in~the~case of~Statement~$(\gamma)$).

It~follows from~Proposition~\ref{ep43} that, for~each $v \in \mathcal{V}$, there exists a~homomorphism~$\sigma_{v}$ of~$G_{v}$ onto~a~$\mathcal{C}$\nobreakdash-group that acts injectively on~all subgroups~$H_{\varepsilon e}$ ($(e,\varepsilon) \in \Theta(v)$). The~family $\mathcal{R} = \{\ker\sigma_{v} \mid v \in \mathcal{V}\}$ is, obviously, a~system of~compatible normal subgroups. Therefore, one can define the~graph of~groups~$\mathcal{G}_{\mathcal{R}}(\mathcal{T})$ and~the~homomorphism $\rho_{\mathcal{R}}\colon \pi_{1}(\mathcal{G}(\mathcal{T})) \to \pi_{1}(\mathcal{G}_{\mathcal{R}}(\mathcal{T}))$, which extends~$\sigma_{v}$ ($v \in \mathcal{V}$) and~takes the~edge subgroups of~$\pi_{1}(\mathcal{G}(\mathcal{T}))$ onto~the~ones of~$\pi_{1}(\mathcal{G}_{\mathcal{R}}(\mathcal{T}))$. Since $\pi_{1}(\mathcal{G}_{\mathcal{R}}(\mathcal{T}))$ satisfies the~conditions of~Statement~$(\alpha)$, there exists a~homomorphism~$\tau$ of~this group that maps it onto~a~$(\mathit{Poly}_{n-1}\mathcal{N}^{\mathfrak{tf}})$-group and~acts injectively on~all subgroups~$H_{\varepsilon e}\rho_{\mathcal{R}}$ ($e \in \mathcal{E}$, $\varepsilon = \pm 1$). It~follows that the~composition~$\rho_{\mathcal{R}}\tau$ is~injective on~the~subgroups~$H_{\varepsilon e}$ ($e \in \mathcal{E}$, $\varepsilon = \pm 1$), and~$\pi_{1}(\mathcal{G}(\mathcal{T}))$ is~residually a~$(\mathit{Poly}_{n+1}\mathcal{C})$-group by~Proposition~\ref{ep62}.
\end{proof}

\begin{proof}[\textbf{\textup{Proof of~Corollary}}]
$(\dag)$\hspace{1ex}Since all~$G_{v}$ ($v \in \mathcal{V}$) are~residually $\mathcal{N}^{\mathfrak{tf}}$\nobreakdash-groups, they are~tor\-sion-free and~the~subgroups~$H_{\varepsilon e}$ ($e \in \mathcal{E}$, $\varepsilon = \pm 1$) have the~same property. If~$\mathcal{N}$ denotes the~class of~all nilpotent groups, then it~follows from~the~arguments used to~prove Statements~$(\beta)$ and~$(\gamma)$ that there exists a~homomorphism of~$\pi_{1}(\mathcal{G}(\mathcal{T}))$ onto~a~$(\mathit{Poly}_{n-1}\mathcal{N}^{\mathfrak{tf}})$-group acting injectively on~all edge subgroups. By~Proposition~\ref{ep46}, the~kernel~$K$ of~this homomorphism splits as~the~free product of~a~free subgroup and~subgroups of~the~form $x^{-1}(K \cap G_{v})x$, where $v \in \mathcal{V}$ and~$x \in \pi_{1}(\mathcal{G}(\mathcal{T}))$. It~is~well known that every free group is~residually an~$\mathcal{N}^{\mathfrak{tf}}$\nobreakdash-group~\cite{Hall1950PAMS, Magnus1935MA} and~the~free product of~any number of~residually $\mathcal{N}^{\mathfrak{tf}}$\nobreakdash-groups is, in~turn, residually an~$\mathcal{N}^{\mathfrak{tf}}$\nobreakdash-group~\cite{Malcev1949MS}. Since~$G_{v}$ ($v \in \mathcal{V}$) are~residually $\mathcal{N}^{\mathfrak{tf}}$\nobreakdash-groups, it~follows from~these facts and~Proposition~\ref{ep43} that any subgroup of~the~form $x^{-1}(K \cap G_{v})x$ ($v \in \mathcal{V}$, $x \in \pi_{1}(\mathcal{G}(\mathcal{T}))$) is~residually an~$\mathcal{N}^{\mathfrak{tf}}$\nobreakdash-group and~$\pi_{1}(\mathcal{G}(\mathcal{T}))$ is~an~extension of~the~residually $\mathcal{N}^{\mathfrak{tf}}$\nobreakdash-group~$K$ by~a~$(\mathit{Poly}_{n-1}\mathcal{N}^{\mathfrak{tf}})$-group. Hence, $\pi_{1}(\mathcal{G}(\mathcal{T}))$ is~residually a~$(\mathit{Poly}_{n}\mathcal{N}^{\mathfrak{tf}})$-group by~Proposition~\ref{ep32}.

$(\ddag)$\hspace{1ex}Let $x \in \pi_{1}(\mathcal{G}(\mathcal{T})) \setminus \{1\}$ be~an~arbitrary element, and~let~$B_{v}$ ($v \in \mathcal{V}$) be~some free bases of~$G_{v}$ ($v \in \mathcal{V}$). Since every locally cyclic subgroup of~a~free group is~cyclic, $\pi_{1}(\mathcal{G}(\mathcal{T}))$ has~cyclic edge subgroups and,~therefore, can~be~given by~a~presentation with~the~set of~generators $S = \bigsqcup_{v \in \mathcal{V}}B_{v}$ and~a~finite set of~defining relations. It~follows that there exists a~finite subset $Q \subseteq S$ containing all generators from~the~relations of~$\pi_{1}(\mathcal{G}(\mathcal{T}))$ and~from~some fixed word defining~$x$.

It~is~obvious that $\pi_{1}(\mathcal{G}(\mathcal{T}))$ splits as~the~free product of~the~subgroups $F_{1} = \operatorname{sgp}\{Q\}$ and~$F_{2} = \operatorname{sgp}\{S \setminus Q\}$. Moreover, $F_{1}$ is~a~tree product of~$n$~free groups of~finite rank with~non-triv\-i\-al cyclic edge subgroups and~$x \in F_{1} \setminus \{1\}$. Let $\rho\colon \pi_{1}(\mathcal{G}(\mathcal{T})) \to F_{1}$ be~the~homomorphism that acts identically on~$F_{1}$ and~sends~$F_{2}$ to~$1$. As~mentioned above, every free group is~residually an~$\mathcal{N}^{\mathfrak{tf}}$\nobreakdash-group. Therefore, $F_{1}$ is~residually a~$(\mathit{Poly}_{n}\mathcal{N}^{\mathfrak{tf}})$-group by~Statement~$(\dag)$. Since it~is~finitely generated, all its homomorphic images have the~same property. Hence, $\rho$ can~be~extended to~a~homomorphism onto~a~finitely generated $(\mathit{Poly}_{n}\mathcal{N}^{\mathfrak{tf}})$\nobreakdash-group taking~$x$ to~a~non-triv\-i\-al element.
\end{proof}

\end{document}

\newtheorem*{rtheorem}{Теорема}
\newtheorem*{rcorollary}{Следствие}
\newtheorem{rproposition}{Предложение}[section]

\theoremstyle{definition}
\newtheorem{rexample}{Пример}

\renewcommand{\theequation}{$\mathfrak{C}_{e}$}
\renewcommand{\thefootnote}{}

\title[О~полинильпотентной аппроксимируемости древесных произведений]{О~полинильпотентной аппроксимируемости древесных произведений групп с~локально циклическими реберными подгруппами}

\author{Е.~В.~Соколов}
\address{Ивановский государственный университет, Россия}
\email{ev-sokolov@yandex.ru}

\begin{abstract}
Пусть $\mathcal{N}$~--- непустой класс нильпотентных групп, замкнутый относительно взятия подгрупп, фак\-тор-групп, конечных прямых произведений и~декартовых степеней вида $\prod_{x \in X}Y_{x}$, где $X \in \mathcal{N}$ и~$Y_{x}$ ($x \in X$)~--- изоморфные копии некоторой группы $Y \in \mathcal{N}$. Пусть также $T$~--- древесное произведение $n$~групп ($2 \leqslant n < \infty$) с~нетривиальными апериодическими локально циклическими реберными подгруппами, $\mathcal{N}^{\mathfrak{tf}}$~--- класс всех $\mathcal{N}$\nobreakdash-групп без~кручения и~$\mathit{Poly}_{k}\mathcal{N}^{\mathfrak{tf}}$ ($k \geqslant 1$)~--- класс групп, обладающих субнормальными рядами длины, не~превосходящей~$k$, с~факторами из~класса~$\mathcal{N}^{\mathfrak{tf}}$. Установлено, что если все свободные и~все вершинные группы $\mathcal{N}^{\mathfrak{tf}}$\nobreakdash-ап\-прок\-си\-ми\-руе\-мы, то~группа~$T$ аппроксимируется $(\mathit{Poly}_{n+1}\mathcal{N}^{\mathfrak{tf}})$-груп\-па\-ми. Доказано также, что если все вершинные группы принадлежат классу~$\mathcal{N}$, то~группа~$T$ обладает гомоморфизмом на~$(\mathit{Poly}_{n-1}\mathcal{N}^{\mathfrak{tf}})$-груп\-пу, инъективным на~всех реберных подгруппах. В~предположении, что $\mathcal{N}$~--- класс всех нильпотентных групп, получен более сильный результат: если все вершинные группы $\mathcal{N}^{\mathfrak{tf}}$\nobreakdash-ап\-прок\-си\-ми\-руе\-мы, то~группа~$T$ аппроксимируется $(\mathit{Poly}_{n}\mathcal{N}^{\mathfrak{tf}})$-груп\-па\-ми. Как~следствие, все так называемые циклически защепленные группы с~одним определяющим соотношением (cy\-cli\-cal\-ly-pinched one-re\-la\-tor groups) аппроксимируются конечно порожденными $(\mathit{Poly}_{2}\mathcal{N}^{\mathfrak{tf}})$-груп\-па\-ми.
\end{abstract}

\keywords{Аппроксимационные свойства, нильпотентная аппроксимируемость, полинильпотентная аппроксимируемость, аппроксимируемость разрешимыми группами, древесное произведение, обобщенное свободное произведение, HNN-рас\-ши\-ре\-ние, предкорневой класс групп}

\thanks{Исследование выполнено в~рамках поддержанного ФГБОУ ВО <<Ивановский государственный университет>> гранта №~2026\nobreakdash-03.}

\vspace*{6pt}\maketitle\vspace{-14pt}

\section{Введение}\label{rsec01}

Пусть $\mathcal{C}$~--- некоторый класс групп (здесь и~далее все рассматриваемые классы групп предполагаются непустыми). Следуя~\cite{Malcev1949MS}, будем говорить, что группа~$X$ \emph{аппроксимируется классом групп~$\mathcal{C}$} (или, более коротко, является \emph{$\mathcal{C}$\nobreakdash-ап\-прок\-си\-ми\-руе\-мой}), если каждый ее неединичный элемент переходит в~неединичный под~действием некоторого гомоморфизма группы~$X$ на~\emph{$\mathcal{C}$\nobreakdash-груп\-пу} (т.\,е.~группу из~класса~$\mathcal{C}$). Через~$\mathcal{C}^{*}(X)$ будем обозначать семейство ядер всех таких гомоморфизмов. Если $\mathcal{C}$~--- это некоторый класс конечных, нильпотентных или~полинильпотентных групп, то~говорят о~\emph{финитной}, \emph{нильпотентной} и~\emph{полинильпотентной аппроксимируемости} соответственно.

Напомним, что согласно одному из~равносильных определений, приведенных в~\cite{Sokolov2015CA}, класс групп~$\mathcal{C}$ называется \emph{корневым}, если он замкнут относительно взятия подгрупп, расширений и~декартовых степеней вида $\prod_{x \in X}Y_{x}$, где $X \in \mathcal{C}$ и~$Y_{x}$ ($x \in X$)~--- изоморфные копии некоторой группы $Y \in \mathcal{C}$. Перечисленные в~этом определении свойства корневых классов позволяют весьма далеко продвинуться в~изучении аппроксимируемости ими свободных конструкций групп: свободных и~древесных произведений, HNN-рас\-ши\-ре\-ний, фундаментальных групп графов групп и~др. (см.,~например,~\cite{AzarovTieudjo2002PIvSU, Tumanova2019SMJ, Sokolov2021SMJ1, Sokolov2021SMJ2, Sokolov2022CA, SokolovTumanova2026JGT, SokolovTumanova2026SMJ}). К~числу конкретных корневых классов, аппроксимируемость которыми активно изучается в~литературе, относятся, например, классы всех конечных групп, конечных $p$\nobreakdash-групп (где $p$~--- простое число), конечных $\mathfrak{P}$\nobreakdash-групп (где $\mathfrak{P}$~--- непустое множество простых чисел), периодических $\mathfrak{P}$\nobreakdash-групп конечного периода, всех разрешимых групп и~разрешимых групп без~кручения.

Вместе с~тем, существуют и~некорневые классы, аппроксимируемость которыми вызывает постоянный интерес. Речь в~первую очередь идет о~классе всех нильпотентных групп и~различных его подклассах. Почти все они не~замкнуты относительно взятия расширений, но~удовлетворяют более слабому свойству: замкнутости относительно взятия конечных прямых произведений (т.\,е.~прямых произведений с~конечным числом сомножителей). Последнее обстоятельство существенно, ведь когда класс~$\mathcal{C}$ содержит подгруппы и~конечные прямые произведения своих групп, на~каждой группе~$X$ можно задать так называемую про-$\mathcal{C}$-то\-по\-ло\-гию, базой окрестностей единицы в~которой служит определенное выше семейство~$\mathcal{C}^{*}(X)$. Про-$\mathcal{C}$-то\-по\-ло\-гии играют важную роль при~изучении аппроксимационных свойств свободных конструкций групп, и~это подсказывает введение следующего понятия.

Класс групп~$\mathcal{C}$ будем называть \emph{предкорневым}, если он замкнут относительно взятия подгрупп, конечных прямых произведений и~декартовых степеней вида $\prod_{x \in X}Y_{x}$, где $X \in \mathcal{C}$ и~$Y_{x}$ ($x \in X$)~--- изоморфные копии некоторой группы $Y \in \mathcal{C}$. Непосредственно из~определений корневого и~предкорневого классов следует, что

1)\hspace{1ex}любое многообразие групп является предкорневым классом;

2)\hspace{1ex}пересечение любого числа корневых и~предкорневых классов вновь оказывается предкорневым классом групп (в~частности, если $\mathcal{C}$~--- предкорневой класс, то~тем~же свойством обладает и~класс всех $\mathcal{C}$\nobreakdash-групп без~кручения);

3)\hspace{1ex}класс, состоящий из~конечных групп, является предкорневым тогда и~только тогда, когда он замкнут относительно взятия подгрупп и~конечных прямых произведений.

Повторяя слово в~слово доказательство предложения~17 из~\cite{SokolovTumanova2020IVM}, можно показать также, что если предкорневой класс~$\mathcal{C}$ состоит из~периодических групп, то~каждая группа из~этого класса имеет конечный период. Конкретными примерами предкорневых классов групп могут служить классы всех нильпотентных групп, конечных нильпотентных групп, нильпотентных групп без~кручения, периодических нильпотентных $\mathfrak{P}$\nobreakdash-групп конечного периода (где $\mathfrak{P}$~--- непустое множество простых чисел), нильпотентных групп, ступени нильпотентности (периоды) которых не~превосходят (соответственно делят) наперед заданное число, а~также всевозможные их пересечения. Еще одним источником примеров предкорневых классов служит следующая конструкция.

Для~любого класса групп~$\mathcal{C}$ через~$\mathit{Poly}_{n}\mathcal{C}$ ($n \geqslant 1$) будем обозначать класс групп, обладающих субнормальными рядами длины, не~превосходящей~$n$, факторы которых принадлежат~$\mathcal{C}$. Очевидно, что если класс~$\mathcal{C}$ замкнут относительно взятия расширений, то~$\mathcal{C} = \mathit{Poly}_{n}\mathcal{C}$ для~любого $n \geqslant 1$. Легко видеть также, что класс~$\mathit{Poly}_{n}\mathcal{C}$ наследует от~класса~$\mathcal{C}$ свойства быть замкнутым относительно взятия фак\-тор-групп и~состоять из~конечных групп, периодических групп или~групп без~кручения. Заметим еще, что если класс~$\mathcal{C}$ замкнут относительно взятия конечных прямых произведений, то~каждая группа из~класса~$\mathit{Poly}_{n}\mathcal{C}$ равномощна некоторой $\mathcal{C}$\nobreakdash-груп\-пе. А~именно, если $X \in \mathit{Poly}_{n}\mathcal{C}$ и~$1 = X_{0} \leqslant X_{1} \leqslant \ldots \leqslant X_{n} = X$~--- субнормальный ряд группы~$X$ с~факторами из~класса~$\mathcal{C}$, то~группа~$X$ равномощна прямому произведению $\mathcal{C}$\nobreakdash-групп~$X_{i}/X_{i-1}$ ($1 \leqslant i \leqslant n$), которое, в~свою очередь, принадлежит классу~$\mathcal{C}$. Отсюда легко следует, что если класс~$\mathcal{C}$ является предкорневым, то~тем~же свойством обладает и~класс~$\mathit{Poly}_{n}\mathcal{C}$. Приводимое ниже предложение~\ref{rp32} показывает, что если $\mathcal{C}$~--- предкорневой класс групп, то~слово <<субнормальными>> в~определении класса~$\mathit{Poly}_{n}\mathcal{C}$ можно заменить словом <<нормальными>>.

В~настоящей статье изучается аппроксимируемость древесных произведений групп классом~$\mathit{Poly}_{n}\mathcal{N}$, где $\mathcal{N}$~--- предкорневой класс нильпотентных групп (определение древесного произведения приводится в~следующем параграфе). Результатов об~аппроксимируемости данной конструкции некорневыми классами нильпотентных и~полинильпотентных групп имеется очень мало. В~случае, когда дерево имеет две вершины, их исчерпывающий обзор приводится в~\cite{Sokolov2024JA}. Там~же доказан ряд утверждений об~аппроксимируемости древесных произведений двух групп конечными нильпотентными и~метанильпотентными группами. Помимо этого известны, по-ви\-ди\-мо\-му, лишь критерии нильпотентной аппроксимируемости фундаментальных групп произвольного графа конечных групп~\cite[теорема~6]{Varsos1996HJM} и~конечного графа бесконечных циклических групп~\cite[теоремы~5 и~6]{Sokolov2021JA}. Настоящая работа дополняет данный перечень достаточным условием аппроксимируемости классом~$\mathit{Poly}_{n}\mathcal{N}$ древесного произведения конечного числа групп с~локально циклическими реберными подгруппами.

\section{Формулировка результатов}\label{rsec02}

Пусть $\mathcal{T}$~--- непустое дерево с~множеством вершин~$\mathcal{V}$ и~множеством ребер~$\mathcal{E}$. Сопоставляя каждой вершине $v \in \mathcal{V}$ некоторую группу~$G_{v}$, а~каждому ребру $e \in \mathcal{E}$~--- группу~$H_{e}$ и~инъективные гомоморфизмы $\varphi_{+e}\colon H_{e} \to G_{e(1)}$, $\varphi_{-e}\colon H_{e} \to G_{e(-1)}$ (где $e(1)$ и~$e(-1)$~--- вершины, являющиеся концами ребра~$e$), получим \emph{граф групп~$\mathcal{G}(\mathcal{T})$} над~$\mathcal{T}$. Все эти обозначения будем предполагать фиксированными до~конца статьи.{\parfillskip=0pt{}\par}

Согласно~\cite{KarrassSolitar1970TAMS} \emph{древесным произведением групп~$G_{v}$} ($v \in \mathcal{V}$) называется фундаментальная группа~$\pi_{1}(\mathcal{G}(\mathcal{T}))$ графа~$\mathcal{G}(\mathcal{T})$, т.\,е.~группа с~представлением, образующими в~котором являются образующие групп~$G_{v}$ ($v \in \mathcal{V}$), а~определяющими соотношениями~--- соотношения групп~$G_{v}$ ($v \in \mathcal{V}$) и~всевозможные соотношения вида $h\varphi_{+e} = h\varphi_{-e}$ ($e \in \mathcal{E}$, $h \in H_{e}$), где $h\varphi_{\varepsilon e}$ ($\varepsilon = \pm 1$)~--- слово в~образующих группы~$G_{e(\varepsilon)}$, задающее образ элемента~$h$ относительно гомоморфизма~$\varphi_{\varepsilon e}$. Группы~$G_{v}$ ($v \in \mathcal{V}$) будем называть \emph{вершинными}, подгруппы $H_{+e} = H_{e}\varphi_{+e}$ и~$H_{-e} = H_{e}\varphi_{-e}$ ($e \in \mathcal{E}$)~--- \emph{реберными}. Известно~\cite[теорема~1]{KarrassSolitar1970TAMS}, что все вершинные группы вкладываются в~группу~$\pi_{1}(\mathcal{G}(\mathcal{T}))$ посредством тождественных отображений образующих и~потому их можно считать ее подгруппами. Основным результатом работы является следующая

\begin{rtheorem}
Пусть $\mathcal{N}$~--- предкорневой класс нильпотентных групп\textup{,} замкнутый относительно взятия фак\-тор-групп\textup{,} и~$\mathcal{N}^{\mathfrak{tf}}$~--- класс всех $\mathcal{N}$\nobreakdash-групп без~кручения. Пусть также $\mathcal{T}$~--- конечное дерево с~$n \geqslant 2$ вершинами и~все подгруппы~$H_{\varepsilon e}$ \textup{(}$e \in \mathcal{E}$\textup{,} $\varepsilon = \pm 1$\textup{)} являются нетривиальными апериодическими локально циклическими группами. Тогда справедливы следующие утверждения.

$(\alpha)$\hspace{1ex}Если все группы~$G_{v}$ \textup{(}$v \in \mathcal{V}$\textup{)} принадлежат классу~$\mathcal{N}$\textup{,} то~группа~$\pi_{1}(\mathcal{G}(\mathcal{T}))$ обладает гомоморфизмом на~$(\mathit{Poly}_{n-1}\mathcal{N}^{\mathfrak{tf}})$-груп\-пу\textup{,} действующим инъективно на~каждой подгруппе~$H_{\varepsilon e}$ \textup{(}$e \in \mathcal{E}$\textup{,} $\varepsilon = \pm 1$\textup{)}.

$(\beta)$\hspace{1ex}Если все группы~$G_{v}$ \textup{(}$v \in \mathcal{V}$\textup{)} и~все свободные группы аппроксимируются классом~$\mathcal{N}$ и~для~любых $e \in \mathcal{E}$\textup{,} $\varepsilon = \pm 1$ группа~$G_{e(\varepsilon)}$ обладает гомоморфизмом на~$\mathcal{N}$\nobreakdash-груп\-пу\textup{,} инъективным на~подгруппе~$H_{\varepsilon e}$\textup{,} то~группа~$\pi_{1}(\mathcal{G}(\mathcal{T}))$ аппроксимируется классом~$\mathit{Poly}_{n+1}\mathcal{N}$.

$(\gamma)$\hspace{1ex}Если все группы~$G_{v}$ \textup{(}$v \in \mathcal{V}$\textup{)} и~все свободные группы аппроксимируются классом~$\mathcal{N}^{\mathfrak{tf}}$\textup{,} то~группа~$\pi_{1}(\mathcal{G}(\mathcal{T}))$ аппроксимируется классом~$\mathit{Poly}_{n+1}\mathcal{N}^{\mathfrak{tf}}$.
\end{rtheorem}

Заметим, что результаты, подобные утверждению~$(\alpha)$, играют важную роль в~изучении аппроксимационных свойств древесных произведений групп, позволяя представить последние в~виде расширений групп с~понятными свойствами. Верхняя часть такого расширения является группой из~наперед заданного класса и~скрывает в~себе почти все сложности, связанные с~графом групп. Нижняя~же часть раскладывается в~свободное произведение некоторой свободной группы и~групп, сопряженных с~подгруппами вершинных групп (см.~предложение~\ref{rp46} ниже), а~для~такой конструкции большинство аппроксимационных свойств исследовано полностью.

Заметим еще, что полинильпотентная группа часто описывается строкой $(c_{1},...,c_{n})$, где $c_{1}$,~\ldots,~$c_{n}$~--- ступени нильпотентности факторов ее полинильпотентного ряда. Если в~утверждении~$(\alpha)$ в~качестве~$\mathcal{N}$ взять класс~$\mathcal{N}_{c}$ всех нильпотентных групп ступени не~выше~$c$, то~образу гомоморфизма, существование которого гарантирует данное утверждение, будет соответствовать строка $(d_{1},\,\ldots,\,d_{n-1})$, где $d_{1},\,\ldots,\,d_{n-1} \leqslant c$. В~утверждениях~$(\beta)$ и~$(\gamma)$ подобный выбор класса~$\mathcal{N}$ уже невозможен. В~самом деле, класс~$\mathcal{N}_{c}$ является многообразием групп, а~любая группа, аппроксимируемая многообразием, как легко видеть, сама принадлежит данному многообразию. Поэтому все свободные группы никак не~могут аппроксимироваться классом~$\mathcal{N}_{c}$ и,~таким образом, утверждения~$(\beta)$ и~$(\gamma)$ ограничивают лишь дл\'ины полинильпотентных рядов гомоморфных образов группы~$\pi_{1}(\mathcal{G}(\mathcal{T}))$, но~не~ступени их факторов.

\begin{rcorollary}
Пусть $\mathcal{N}^{\mathfrak{tf}}$~--- класс всех нильпотентных групп без~кручения\textup{,} $\mathcal{T}$~--- конечное дерево с~$n \geqslant 2$ вершинами и~каждая подгруппа~$H_{\varepsilon e}$ \textup{(}$e \in \mathcal{E}$\textup{,} $\varepsilon = \pm 1$\textup{)} является нетривиальной локально циклической группой. Тогда справедливы следующие утверждения.

$(\dag)$\hspace{1ex}Если все группы~$G_{v}$ \textup{(}$v \in \mathcal{V}$\textup{)} $\mathcal{N}^{\mathfrak{tf}}$\nobreakdash-ап\-прок\-си\-ми\-руе\-мы\textup{,} то~группа~$\pi_{1}(\mathcal{G}(\mathcal{T}))$ аппроксимируется $(\mathit{Poly}_{n}\mathcal{N}^{\mathfrak{tf}})$-груп\-па\-ми.

$(\ddag)$\hspace{1ex}Если все группы~$G_{v}$ \textup{(}$v \in \mathcal{V}$\textup{)} являются свободными\textup{,} то~группа~$\pi_{1}(\mathcal{G}(\mathcal{T}))$ аппроксимируется конечно порожденными $(\mathit{Poly}_{n}\mathcal{N}^{\mathfrak{tf}})$-груп\-па\-ми. В~частности\textup{,} все циклически защепленные группы с~одним определяющим соотношением \textup{(}т.\,е.~свободные произведения двух свободных групп с~циклическим объединением\textup{)} аппроксимируются конечно порожденными $(\mathit{Poly}_{2}\mathcal{N}^{\mathfrak{tf}})$\nobreakdash-груп\-па\-ми.
\end{rcorollary}

Дл\'ины полинильпотентных рядов гомоморфных образов группы~$\pi_{1}(\mathcal{G}(\mathcal{T}))$ в~теореме и~следствии ограничены числами $n-1$, $n$ и~$n+1$. Вопрос о~точности этих оценок в~общем случае остается открытым. Приводимое ниже предложение~\ref{rp48} и~основные результаты работ~\cite{BaumslagG2010IJM, Labute2015Arxiv} показывают, что при~определенных дополнительных ограничениях указанные оценки могут быть уменьшены. Вместе с~тем, известно, что не~все циклически защепленные группы с~одним определяющим соотношением аппроксимируются нильпотентными группами~\cite{Azarov1998MN}. Поэтому при~$n=2$ полученная в~следствии оценка является точной.

Если забыть о~полинильпотентности и~считать, что в~сформулированных выше утверждениях речь идет о~гомоморфизмах на~разрешимые группы, то~естественно спросить, справедливы~ли указанные утверждения в~случае, когда дерево~$\mathcal{T}$ бесконечно. Отрицательный ответ на~данный вопрос дает

\begin{rexample}\label{rex01}
Пусть дерево~$\mathcal{T}$ представляет собой бесконечную цепь, вершины которой занумерованы целыми числами. Пусть также вершине с~номером~$i$ сопоставлена свободная группа~$F_{i}$ с~базисом $\{a_{i},b_{i}\}$, а~ребру, соединяющему вершины с~номерами~$i$ и~$i+1$,~--- бесконечная циклическая группа и~отображения, переводящие ее образующий в~элементы $[b_{i}^{\vphantom{1}}, a_{i}^{-1}b_{i}^{\vphantom{1}}a_{i}^{\vphantom{1}}]$ и~$b_{i+1}^{\vphantom{1}}$. Тогда группа~$\pi_{1}(\mathcal{G}(\mathcal{T}))$ имеет представление $\langle a_{i}^{\vphantom{1}}, b_{i}^{\vphantom{1}};\ b_{i+1}^{\vphantom{1}} = [b_{i}^{\vphantom{1}}, a_{i}^{-1}b_{i}^{\vphantom{1}}a_{i}^{\vphantom{1}}]\ (i \in \mathbb{Z}) \rangle$, из~которого следует, что для~любых $i \in \mathbb{Z}$, $k \geqslant 1$ элемент~$b_{i+k}$ принадлежит $k$\nobreakdash-му~члену~$(\pi_{1}(\mathcal{G}(\mathcal{T})))^{(k)}$ ряда коммутантов данной группы. Так как $b_{i} = b_{(i-k)+k}$ при~всех $i \in \mathbb{Z}$, $k \geqslant 1$, то~$b_{i} \in \bigcap_{k \geqslant 1}(\pi_{1}(\mathcal{G}(\mathcal{T})))^{(k)}$ для~каждого $i \in \mathbb{Z}$. Таким образом, ядро любого гомоморфизма группы~$\pi_{1}(\mathcal{G}(\mathcal{T}))$ на~разрешимую группу содержит все реберные подгруппы и,~в~частности, группа~$\pi_{1}(\mathcal{G}(\mathcal{T}))$ не~аппроксимируется разрешимыми группами. Пусть теперь для~каждого $i \in \mathbb{Z}$
\[
N_{i}^{\vphantom{1}} = \langle 
a_{i}^{\vphantom{1}}, b_{i}^{\vphantom{1}}, 
c_{i}^{\vphantom{1}}, d_{i}^{\vphantom{1}};\ 
a_{i}^{-1}b_{i}^{\vphantom{1}}a_{i}^{\vphantom{1}} = 
b_{i}^{\vphantom{1}}c_{i}^{\vphantom{1}},\ 
b_{i}^{-1}c_{i}^{\vphantom{1}}b_{i}^{\vphantom{1}} = 
c_{i}^{\vphantom{1}}d_{i}^{\vphantom{1}},\ 
[a_{i}^{\vphantom{1}},c_{i}^{\vphantom{1}}] = 
[a_{i}^{\vphantom{1}},d_{i}^{\vphantom{1}}] = 
[b_{i}^{\vphantom{1}},d_{i}^{\vphantom{1}}] = 
[c_{i}^{\vphantom{1}},d_{i}^{\vphantom{1}}] = 1
\rangle.
\]
Тогда $1 < \langle d_{i} \rangle < \langle c_{i},d_{i} \rangle < \langle b_{i},c_{i},d_{i} \rangle < N_{i}$~--- центральный ряд группы~$N_{i}$ с~бесконечными циклическими факторами и~$[b_{i}^{\vphantom{1}}, a_{i}^{-1}b_{i}^{\vphantom{1}}a_{i}^{\vphantom{1}}] = [b_{i}^{\vphantom{1}}, b_{i}^{\vphantom{1}}c_{i}^{\vphantom{1}}] = [b_{i}^{\vphantom{1}}, c_{i}^{\vphantom{1}}] = d_{i}^{-1} \ne 1$. Заменяя в~построенном выше древесном произведении группы~$F_{i}$ ($i \in \mathbb{Z}$) на~$N_{i}$ ($i \in \mathbb{Z}$), получаем, что условие конечности дерева в~возможном аналоге утверждения~$(\alpha)$ существенно даже в~том случае, когда все сомножители имеют одинаковые ступени нильпотентности.
\end{rexample}

Теорему и~следствие нельзя распространить и~на~конструкцию HNN-рас\-ши\-ре\-ния, как показывает

\begin{rexample}\label{rex02}
Пусть символ~$B_{0}$ обозначает группу~$F_{0}$ или~$N_{0}$ из~предыдущего примера, $E$~--- HNN-рас\-ши\-ре\-ние группы~$B_{0}$ с~проходной буквой~$t$ и~связанными циклическими подгруппами, порожденными элементами~$b_{0}^{\vphantom{1}}$ и~$[b_{0}^{\vphantom{1}}, a_{0}^{-1}b_{0}^{\vphantom{1}}a_{0}^{\vphantom{1}}]$. Используя метод Рей\-де\-мей\-с\-те\-ра~--- Шрейе\-ра (см.,~например,~\cite[\S\,2.3]{MagnusKarrasSolitar1976}), нетрудно показать, что нормальное замыкание подгруппы~$B_{0}$ в~группе~$E$ изоморфно древесному произведению из~примера~\ref{rex01}, причем связанные подгруппы группы~$E$ переходят при~этом изоморфизме на~реберные подгруппы указанного древесного произведения. Отсюда следует, что ядро любого гомоморфизма группы~$E$ на~разрешимую группу содержит обе ее связанные подгруппы и,~в~частности, группа~$E$ не~аппроксимируется разрешимыми группами.
\end{rexample}

Оставшаяся часть статьи посвящена доказательству теоремы и~следствия.

\section{Некоторые свойства класса~\texorpdfstring{$\mathit{Poly}_{n}\mathcal{C}$}{Poly\_\{n\}C}}\label{rsec03}

Если $\mathcal{C}$ и~$\mathcal{D}$~--- произвольные классы групп, то~через~$\mathcal{C} \cdot \mathcal{D}$ будем обозначать класс групп, состоящий из~всевозможных расширений $\mathcal{C}$\nobreakdash-груп\-пы при~помощи $\mathcal{D}$\nobreakdash-груп\-пы.

\begin{rproposition}\label{rp31}
Пусть $\mathcal{C}$~--- предкорневой класс групп\textup{,} $X$~--- расширение группы~$Y$\kern-2.5pt{} и~мощность фак\-тор-груп\-пы~$X/Y$\kern-2.5pt{} не~превосходит мощности некоторой $\mathcal{C}$\nobreakdash-груп\-пы~$G$. Тогда справедливы следующие утверждения.

\textup{1.}\hspace{1ex}Любая подгруппа $Z \in \mathcal{C}^{*}(Y)$ содержит подгруппу $U \in \mathcal{C}^{*}(Y)$\textup{,} нормальную в~$X$.

\textup{2.}\hspace{1ex}Если группа~$Y$ $\mathcal{C}$\nobreakdash-ап\-прок\-си\-ми\-руе\-ма\textup{,} то~группа~$X$ аппроксимируется классом $\mathcal{C} \cdot \{X/Y\}$\textup{,} где $\{X/Y\}$~--- класс\textup{,} состоящий из~единственной группы~$X/Y$.
\end{rproposition}

\begin{proof}
1.\hspace{1ex}Пусть $S$~--- некоторая полная система представителей смежных классов группы~$X$ по~подгруппе~$Y$ и~$U = \bigcap_{s \in S}s^{-1}Zs$. Поскольку подгруппа~$Z$ нормальна в~группе~$Y$, справедливы равенства $U = \bigcap_{s \in S,\,y \in Y}s^{-1}y^{-1}Zys = \bigcap_{x \in X}x^{-1}Zx$, из~которых следует, что подгруппа~$U$ нормальна в~группе~$X$. По~теореме Ремака (см.,~например,~\cite[теорема~4.3.9]{KargapolovMerzlyakov1982}) фак\-тор-груп\-па~$Y/U$ вкладывается в~декартово произведение~$P$ фак\-тор-групп $Y/s^{-1}Zs$ ($s \in S$), изоморфных $\mathcal{C}$\nobreakdash-груп\-пе~$Y/Z$. Так как мощность множества~$S$ совпадает с~мощностью фак\-тор-груп\-пы~$X/Y$ и~не~превосходит мощности группы $G \in \mathcal{C}$, группа~$P$, в~свою очередь, вкладывается в~декартово произведение $Q = \prod_{g \in G}Y_{g}$, где $Y_{g}$ ($g \in G$)~--- изоморфные копии $\mathcal{C}$\nobreakdash-груп\-пы~$Y/Z$. Поскольку класс~$\mathcal{C}$ является предкорневым, отсюда следует, что $Q \in \mathcal{C}$, $P \in \mathcal{C}$ и~$Y/U \in \mathcal{C}$, что и~требовалось.

2.\hspace{1ex}Пусть $x \in X \setminus \{1\}$. Укажем гомоморфизм группы~$X$ на~$(\mathcal{C} \cdot \{X/Y\})$-груп\-пу, переводящий~$x$ в~неединичный элемент. Если $x \notin Y$, то~искомым оказывается естественный гомоморфизм $X \to X/Y$. Если $x \in Y$, воспользуемся $\mathcal{C}$\nobreakdash-ап\-прок\-си\-ми\-руе\-мо\-стью группы~$Y$ и~выберем подгруппу $Z \in \mathcal{C}^{*}(Y)$, не~содержащую элемента~$x$. Согласно утверждению~1 существует подгруппа $U \in \mathcal{C}^{*}(Y)$, нормальная в~$X$ и~лежащая в~$Z$. Тогда $x \notin U$ и~естественный гомоморфизм $X \to X/U$ является искомым.
\end{proof}

Предваряя формулировку следующего предложения, напомним, что всякий корневой класс~$\mathcal{C}$ удовлетворяет \emph{условию Грюнберга}: \pagebreak если $1 \leqslant Z \leqslant Y \leqslant X$~--- субнормальный ряд некоторой группы~$X$, то~из~включений $X/Y \in \mathcal{C}$ и~$Y/Z \in \mathcal{C}$ вытекает существование подгруппы $U \in \mathcal{C}^{*}(X)$, лежащей~в~$Z$~\cite{Sokolov2015CA}.

\begin{rproposition}\label{rp32}
Если $\mathcal{C}$~--- предкорневой класс групп\textup{,} то~справедливы следующие утверждения.

\textup{1.}\hspace{1ex}Пусть $X_{1} \leqslant \ldots \leqslant X_{n} = X$ \textup{(}$n \geqslant 1$\textup{)}~--- субнормальная последовательность подгрупп группы~$X$ с~факторами из~класса~$\mathcal{C}$. Тогда существует последовательность нормальных подгрупп $Y_{1} \leqslant \ldots \leqslant Y_{n} = X$ такая\textup{,} что $Y_{i+1}/Y_{i} \in \mathcal{C}$ и~$Y_{i} \leqslant X_{i}$ для~всех $i \in \{1,\,\ldots,\,n-1\}$. Как следствие\textup{,} группа~$X$ принадлежит классу~$\mathit{Poly}_{n}\mathcal{C}$ \textup{(}$n \geqslant 1$\textup{)} тогда и~только тогда\textup{,} когда она обладает нормальным рядом длины\textup{,} не~превосходящей~$n$\textup{,} с~факторами из~класса~$\mathcal{C}$.

\textup{2.}\hspace{1ex}Пусть $1 \leqslant Z \leqslant Y \leqslant X$~--- субнормальный ряд группы~$X$. Если $X/Y \in \mathit{Poly}_{m}\mathcal{C}$ и~$Y/Z \in \mathit{Poly}_{n}\mathcal{C}$ для~некоторых $m,n \geqslant 1$\textup{,} то~существует лежащая в~$Z$ подгруппа $U \in (\mathit{Poly}_{m+n}\mathcal{C})^{*}(X)$. В~частности\textup{,} класс~$\mathcal{C}$ удовлетворяет \emph{ослабленному условию Грюнберга}\textup{:} если $X/Y \in \mathcal{C}$ и~$Y/Z \in \mathcal{C}$\textup{,} то~существует лежащая~в~$Z$ подгруппа $U \in (\mathcal{C} \cdot \mathcal{C})^{*}(X)$.

\textup{3.}\hspace{1ex}Пусть $X$~--- расширение группы~$Y$\kern-2pt{} при~помощи группы из~класса~$\mathit{Poly}_{m}\mathcal{C}$ \textup{(}$m\kern-.5pt{} \geqslant\kern-.5pt{}\nolinebreak 1$\textup{)}. Если группа~$Y$ аппроксимируется классом~$\mathit{Poly}_{n}\mathcal{C}$ для~некоторого $n \geqslant 1$\textup{,} то~группа~$X$ аппроксимируется классом~$\mathit{Poly}_{m+n}\mathcal{C}$.
\end{rproposition}

\begin{proof}
1.\hspace{1ex}Понятно, что в~качестве~$Y_{n-1}$ можно взять подгруппу~$X_{n-1}$. Предположим, что подгруппа~$Y_{i}$ ($2 \leqslant i \leqslant n-1$) уже определена. Тогда $Y_{i}/Y_{i} \cap X_{i-1} \cong Y_{i}X_{i-1}/X_{i-1} \leqslant X_{i}/X_{i-1} \in \mathcal{C}$ и,~поскольку класс~$\mathcal{C}$ замкнут относительно взятия подгрупп, $Y_{i}/Y_{i} \cap X_{i-1} \in \mathcal{C}$. Как было отмечено во~введении, $(\mathit{Poly}_{n-i}\mathcal{C})$-груп\-па~$X/Y_{i}$ равномощна некоторой $\mathcal{C}$\nobreakdash-груп\-пе. Отсюда и~из~предложения~\ref{rp31} (применяемого к~подгруппам~$Y_{i}$ и~$X_{i-1} \cap Y_{i}$) следует, что найдется подгруппа $Y_{i-1} \in \mathcal{C}^{*}(Y_{i})$, нормальная в~$X$ и~лежащая в~$X_{i-1} \cap Y_{i}$. Очевидно, что эта подгруппа является искомой.

2.\hspace{1ex}Поскольку $X/Y \in \mathit{Poly}_{m}\mathcal{C}$ и~$Y/Z \in \mathit{Poly}_{n}\mathcal{C}$, в~группе~$X$ имеется субнормальная последовательность подгрупп $Z \leqslant Y_{1} \leqslant \ldots \leqslant Y_{n} = Y \leqslant X_{1} \leqslant \ldots \leqslant X_{m} = X$ с~факторами из~класса~$\mathcal{C}$. Применяя к~ней утверждение~1, получаем искомую подгруппу~$U$.{\parfillskip=0pt{}\par}

3.\hspace{1ex}Так как $(\mathit{Poly}_{m}\mathcal{C})$-груп\-па~$X/Y$ равномощна некоторой $\mathcal{C}$\nobreakdash-груп\-пе, $\mathcal{C} \subseteq \mathit{Poly}_{n}\mathcal{C}$ и~класс~$\mathit{Poly}_{n}\mathcal{C}$ является предкорневым, то~согласно предложению~\ref{rp31} группа~$X$ аппроксимируется классом $\mathit{Poly}_{n}\mathcal{C} \cdot \{X/Y\} \subseteq \mathit{Poly}_{n}\mathcal{C} \cdot \mathit{Poly}_{m}\mathcal{C} \subseteq \mathit{Poly}_{m+n}\mathcal{C}$.
\end{proof}

\section{Некоторые вспомогательные утверждения}\label{rsec04}

Всюду далее будем использовать следующие обозначения:

\noindent
\makebox[5.5ex][l]{$\ker\sigma$}~--- ядро гомоморфизма~$\sigma$;

\noindent
\makebox[5.5ex][l]{$\mathfrak{T}(X)$}~--- множество всех элементов конечного порядка нильпотентной группы~$X$;

\noindent
\makebox[5.5ex][l]{$\zeta_{i}X$}~--- $i$\nobreakdash-й~член верхнего центрального ряда нильпотентной группы~$X$.

\begin{rproposition}\label{rp41}
Для~любой нильпотентной группы~$X$ справедливы следующие утверждения.

\textup{1.}\hspace{1ex}Множество~$\mathfrak{T}(X)$ является нормальной подгруппой группы~$X$~\textup{\cite[теорема~5.2]{ClementMajewiczZyman2017}}.

\textup{2.}\hspace{1ex}Если группа~$X$ не~имеет кручения\textup{,} то~все члены ее верхнего центрального ряда изолированы~\textup{\cite[лемма~5.9]{ClementMajewiczZyman2017}}.
\end{rproposition}

\begin{rproposition}\label{rp42}
Пусть $X$~--- нильпотентная группа без~кручения и~$Y$~--- не\-три\-ви\-аль\-ная локально циклическая подгруппа группы~$X$. Тогда существует число $\lambda(Y) \geqslant 1$ такое\textup{,} что $Y \leqslant \zeta_{\lambda(Y)}X$ и~$Y \cap \zeta_{\lambda(Y)-1}X = 1$.
\end{rproposition}

\begin{proof}
Поскольку группа~$X$ нильпотентна и~$Y \ne 1$, найдутся элемент $y \in Y \setminus \{1\}$ и~число $\lambda(Y) \geqslant\nolinebreak 1$, \pagebreak удовлетворяющие условию $y \in \zeta_{\lambda(Y)}X \setminus \zeta_{\lambda(Y)-1}X$. Пусть $y^{\prime} \in Y \cap \zeta_{\lambda(Y)-1}X$. Тогда элементы~$y$ и~$y^{\prime}$ порождают циклическую подгруппу~$Z$ группы~$Y$. Так как группа~$X$ не~имеет кручения, то~подгруппа~$\zeta_{\lambda(Y)-1}X$ изолирована в~ней в~силу предложения~\ref{rp41}. Отсюда и~из~соотношения $y \in Z \setminus \zeta_{\lambda(Y)-1}X$ следует, что $Z \cap \zeta_{\lambda(Y)-1}X = 1$ и~$y^{\prime} = 1$. Значит, $Y \cap \zeta_{\lambda(Y)-1}X = 1$. Аналогичным образом из~существования элемента $y^{\prime\prime} \in Y \cap \zeta_{\lambda(Y)+1}X \setminus \zeta_{\lambda(Y)}X$ вытекает равенство $\operatorname{sgp}\{y,y^{\prime\prime}\} \cap\nolinebreak \zeta_{\lambda(Y)}X = 1$, противоречащее выбору элемента~$y$. Стало быть, $Y \leqslant \zeta_{\lambda(Y)}X$.
\end{proof}

\begin{rproposition}\label{rp43}
\textup{\cite[предложение~3.1]{Sokolov2024JA}}
Пусть $\mathcal{C}$~--- класс групп\textup{,} замкнутый относительно взятия подгрупп и~конечных прямых произведений\textup{,} $X$~--- некоторая группа\textup{,} $Y$~и~$Z$~--- подгруппы группы~$X$. Тогда справедливы следующие \mbox{утверждения}.

\textup{1.}\hspace{1ex}Если $Y \in \mathcal{C}^{*}(X)$\textup{,} то~$Y \cap Z \in \mathcal{C}^{*}(Z)$. Как следствие\textup{,} любая подгруппа $\mathcal{C}$\nobreakdash-ап\-прок\-си\-ми\-руе\-мой группы сама является $\mathcal{C}$\nobreakdash-ап\-прок\-си\-ми\-руе\-мой.

\textup{2.}\hspace{1ex}Если $Y,Z \in \mathcal{C}^{*}(X)$\textup{,} то~$Y \cap Z \in \mathcal{C}^{*}(X)$. Как следствие\textup{,} для~любого конечного множества нетривиальных элементов $\mathcal{C}$\nobreakdash-ап\-прок\-си\-ми\-руе\-мой группы найдется непересекающаяся с~этим множеством подгруппа из~семейства~$\mathcal{C}^{*}(X)$.
\end{rproposition}

Говорят (см.,~например,~\cite[\S\,2.2]{DixonKurdachenkoSubbotin2017}), что группа имеет \emph{конечный $0$\nobreakdash-ранг}, если она обладает конечным субнормальным рядом, каждый фактор которого является либо~периодической, либо~бесконечной циклической группой.

\begin{rproposition}\label{rp44}
\textup{\cite[предложение~18]{Tumanova2015IVM}}
Пусть $\mathcal{C}$~--- класс групп без~кручения\textup{,} замкнутый относительно взятия подгрупп и~конечных прямых произведений\textup{,} $X$~--- $\mathcal{C}$\nobreakdash-ап\-прок\-си\-ми\-руе\-мая группа и~$Y$~--- подгруппа группы~$X$\textup{,} имеющая конечный $0$\nobreakdash-ранг. Тогда существует подгруппа $Z \in \mathcal{C}^{*}(X)$ такая\textup{,} что $Z \cap Y = 1$.
\end{rproposition}

Следующее утверждение вытекает из~результатов статьи~\cite{Cohen1974JAMS}.

\begin{rproposition}\label{rp45}
Пусть дерево~$\mathcal{T}$ содержит одно ребро~$e$ и~$N$~--- нормальная подгруппа группы~$\pi_{1}(\mathcal{G}(\mathcal{T}))$\textup{,} тривиально пересекающаяся с~подгруппами~$H_{+e}$ и~$H_{-e}$. Тогда подгруппа~$N$ раскладывается в~свободное произведение некоторой свободной подгруппы\textup{,} подгрупп $N \cap G_{e(1)}$ и~$N \cap G_{e(-1)}$\textup{,} а~также некоторых подгрупп вида $x^{-1}(N \cap G_{e(\varepsilon)})x$\textup{,} где $x \in \pi_{1}(\mathcal{G}(\mathcal{T})) \setminus \{1\}$\textup{,} $\varepsilon = \pm 1$.
\end{rproposition}

С~помощью индукции по~числу вершин из~предложения~\ref{rp45} получается

\begin{rproposition}\label{rp46}
Пусть дерево~$\mathcal{T}$ конечно и~$N$~--- нормальная подгруппа группы~$\pi_{1}(\mathcal{G}(\mathcal{T}))$\textup{,} тривиально пересекающаяся с~каждой подгруппой~$H_{\varepsilon e}$ \textup{(}$e \in \mathcal{E}$\textup{,} $\varepsilon = \pm 1$\textup{)}. Тогда подгруппа~$N$ раскладывается в~свободное произведение некоторой свободной подгруппы\textup{,} подгрупп~$N \cap G_{v}$ \textup{(}$v \in \mathcal{V}$\textup{),} а~также некоторых подгрупп вида $x^{-1}(N \cap G_{v})x$ \textup{(}$v \in \mathcal{V}$\textup{,} $x \in \pi_{1}(\mathcal{G}(\mathcal{T})) \setminus \{1\}$\textup{)}.
\end{rproposition}

Легко видеть, что если для~любых $e \in \mathcal{E}$, $\varepsilon = \pm 1$ подгруппа~$H_{\varepsilon e}$ лежит в~центре группы~$G_{e(\varepsilon)}$, то~$S = \operatorname{sgp}\{h\varphi_{+e}(h\varphi_{-e})^{-1} \mid e \in \mathcal{E},\ h \in H_{e}\}$~--- нормальная подгруппа прямого произведения~$D$ групп~$G_{v}$ ($v \in \mathcal{V}$). Следуя~\cite{SokolovTumanova2019AL}, фак\-тор-груп\-пу $\mathfrak{D}(\mathcal{G}(\mathcal{T})) = D/S$ будем называть \emph{обобщенным прямым произведением\textup{,} ассоциированным с~графом групп~$\mathcal{G}(\mathcal{T})$}.

\begin{rproposition}\label{rp47}
\textup{\cite[теорема~1]{SokolovTumanova2019AL}}
Если для~любых $e \in \mathcal{E}$\textup{,} $\varepsilon = \pm 1$ подгруппа~$H_{\varepsilon e}$ лежит в~центре группы~$G_{e(\varepsilon)}$\textup{,} то~каждая группа~$G_{v}$ \textup{(}$v \in \mathcal{V}$\textup{)} вкладывается в~группу~$\mathfrak{D}(\mathcal{G}(\mathcal{T}))$ посредством тождественного отображения образующих.
\end{rproposition}

Очевидно, что тождественное отображение образующих группы~$\pi_{1}(\mathcal{G}(\mathcal{T}))$ в~группу~$\mathfrak{D}(\mathcal{G}(\mathcal{T}))$, будучи продолженным до~отображения слов, переводит все определяющие соотношения группы~$\pi_{1}(\mathcal{G}(\mathcal{T}))$ в~равенства, верные в~группе~$\mathfrak{D}(\mathcal{G}(\mathcal{T}))$, и~потому задает некоторый гомоморфизм~$\sigma$. Из~коммутативности для~каждой вершины $v \in \mathcal{V}$ диаграммы
\[
\begin{CD}
G_{v} @>{\iota}>> \pi_{1}(\mathcal{G}(\mathcal{T})) \\
@| @VV{\sigma}V \\ 
G_{v} @>{\iota}>> \mathfrak{D}(\mathcal{G}(\mathcal{T}))
\end{CD}
\]
(где символ~$\iota$ обозначает вложение посредством тождественного отображения образующих) следует, что гомоморфизм~$\sigma$ действует инъективно на~всех подгруппах~$G_{v}$ ($v \in \mathcal{V}$).

\begin{rproposition}\label{rp48}
Пусть $\mathcal{N}$~--- предкорневой класс нильпотентных групп\textup{,} замкнутый относительно взятия фак\-тор-групп\textup{,} и~$\mathcal{N}^{\mathfrak{tf}}$~--- класс всех $\mathcal{N}$\nobreakdash-групп без~кручения. Пусть также дерево~$\mathcal{T}$ конечно\textup{,} все группы~$G_{v}$ \textup{(}$v \in \mathcal{V}$\textup{)} принадлежат классу~$\mathcal{N}^{\mathfrak{tf}}$ и~для~любых $e \in \mathcal{E}$\textup{,} $\varepsilon = \pm 1$ подгруппа~$H_{\varepsilon e}$ лежит в~центре группы~$G_{e(\varepsilon)}$. Тогда существует гомоморфизм группы~$\pi_{1}(\mathcal{G}(\mathcal{T}))$ на~$\mathcal{N}^{\mathfrak{tf}}$\nobreakdash-груп\-пу\textup{,} действующий инъективно на~группах~$G_{v}$ \textup{(}$v \in \mathcal{V}$\textup{)}.
\end{rproposition}

\begin{proof}
Пусть $\mathfrak{D}(\mathcal{G}(\mathcal{T}))$~--- обобщенное прямое произведение, ассоциированное с~графом групп~$\mathcal{G}(\mathcal{T})$, и~$\sigma\colon \pi_{1}(\mathcal{G}(\mathcal{T})) \to \mathfrak{D}(\mathcal{G}(\mathcal{T}))$~--- гомоморфизм, продолжающий тождественное отображение образующих группы~$\pi_{1}(\mathcal{G}(\mathcal{T}))$. Тогда $\mathfrak{D}(\mathcal{G}(\mathcal{T})) \in \mathcal{N}$ ввиду замкнутости класса~$\mathcal{N}$ относительно взятия конечных прямых произведений и~фак\-тор-групп. Так как группы~$G_{v}$ ($v \in \mathcal{V}$) не~имеют кручения, то~естественный гомоморфизм $\eta\colon \mathfrak{D}(\mathcal{G}(\mathcal{T})) \to \mathfrak{D}(\mathcal{G}(\mathcal{T})) / \mathfrak{T}(\mathfrak{D}(\mathcal{G}(\mathcal{T})))$ действует на~них инъективно. Поскольку тем~же свойством обладает и~гомоморфизм~$\sigma$, композиция~$\sigma\eta$ является искомым отображением.
\end{proof}

\section{Доказательство утверждения~\texorpdfstring{$(\alpha)$}{(alpha)}}\label{rsec05}

Далее для~каждой вершины $v \in \mathcal{V}$ символом~$\Theta(v)$ будем обозначать множество $\{(e,\varepsilon) \mid e \in \mathcal{E},\ \varepsilon = \pm 1,\ v = e(\varepsilon)\}$.

\begin{rproposition}\label{rp51}
Пусть $\mathcal{C}$~--- предкорневой класс групп\textup{,} дерево~$\mathcal{T}$ конечно и~для каждой вершины $v \in \mathcal{V}$ определено число $n_{v} \geqslant 1$ такое\textup{,} что $G_{v} \in \mathit{Poly}_{n_{v}}\mathcal{C}$. Пусть также для~некоторого $n \geqslant 1$ группа~$\pi_{1}(\mathcal{G}(\mathcal{T}))$ обладает гомоморфизмом~$\sigma$ на~группу из~класса~$\mathit{Poly}_{n}\mathcal{C}$\textup{,} действующим инъективно на~подгруппах~$H_{\varepsilon e}$ \textup{(}$e \in \mathcal{E}$\textup{,} $\varepsilon = \pm 1$\textup{)}. Тогда для~любых вершины $u \in \mathcal{V}$ и~подгруппы $K \leqslant G_{u}$ существует гомоморфизм группы~$\pi_{1}(\mathcal{G}(\mathcal{T}))$ на~группу из~класса~$\mathit{Poly}_{n+n_{u}}\mathcal{C}$\textup{,} действующий инъективно на~подгруппах~$H_{\varepsilon e}$ \textup{(}$e \in \mathcal{E}$\textup{,} $\varepsilon = \pm 1$\textup{)} и~подгруппе~$K$.
\end{rproposition}

\begin{proof}
Согласно предложению~\ref{rp46} группа $N = \ker\sigma$ раскладывается в~свободное произведение, одним из~сомножителей которого служит подгруппа~$N \cap G_{u}$. Поэтому существует гомоморфизм $\tau\colon N \to N \cap G_{u}$, действующий тождественно на~этой подгруппе. Так как $\pi_{1}(\mathcal{G}(\mathcal{T})) / N \in \mathit{Poly}_{n}\mathcal{C}$, $G_{u} \in \mathit{Poly}_{n_{u}}\mathcal{C}$ и~класс~$\mathit{Poly}_{n_{u}}\mathcal{C}$ замкнут относительно взятия подгрупп, то~$N \cap G_{u} \in \mathit{Poly}_{n_{u}}\mathcal{C}$ и~в~силу предложения~\ref{rp32}, применяемого к~ряду $1 \leqslant M \leqslant N \leqslant \pi_{1}(\mathcal{G}(\mathcal{T}))$, где $M = \ker\tau$, найдется подгруппа $L \in (\mathit{Poly}_{n+n_{u}}\mathcal{C})^{*}(\pi_{1}(\mathcal{G}(\mathcal{T})))$, лежащая в~$M$. Пусть $k \in K \setminus \{1\}$. Если $k \notin N$, то~$k \notin L$, поскольку $L \leqslant M \leqslant N$. Если $k \in N$, то~$k \in N \cap G_{u}$ и,~так как гомоморфизм~$\tau$ действует тождественно на~подгруппе~$N \cap G_{u}$, справедливы соотношения $k \notin M$ и~$k \notin L$. Значит, $K \cap L = 1$. Из~определения гомоморфизма~$\sigma$ следует также, что $H_{\varepsilon e} \cap L = 1$ для~любых $e \in \mathcal{E}$, $\varepsilon = \pm 1$. Стало быть, естественный гомоморфизм $\pi_{1}(\mathcal{G}(\mathcal{T})) \to \pi_{1}(\mathcal{G}(\mathcal{T}))/L$ является искомым.
\end{proof}

Пусть для~каждой вершины $v \in \mathcal{V}$ в~группе~$G_{v}$ зафиксирована некоторая нормальная подгруппа~$R_{v}$ и~для~любого ребра $e \in \mathcal{E}$ справедливо равенство
\begin{equation}\label{eq}
(R_{e(1)} \cap H_{+e})(\varphi_{+e})^{-1} = (R_{e(-1)} \cap H_{-e})(\varphi_{-e})^{-1}.
\end{equation}
Тогда семейство $\mathcal{R} = \{R_{v} \mid v \in \mathcal{V}\}$ будем называть \emph{системой совместимых нормальных подгрупп в~графе групп~$\mathcal{G}(\mathcal{T})$}. Положим $\overline{G}_{v} = G_{v}/R_{v}$ ($v \in \mathcal{V}$)~и
\[
\overline{H}_{e} = 
H_{e}/(R_{e(1)} \cap H_{+e})(\varphi_{+e})^{-1} = 
H_{e}/(R_{e(-1)} \cap H_{-e})(\varphi_{-e})^{-1}
\]
($e \in \mathcal{E}$, $\varepsilon = \pm 1$). Из~равенств~\eqref{eq} ($e \in \mathcal{E}$) следует, что для~любых $e \in \mathcal{E}$, $\varepsilon = \pm 1$ отображение $\overline{\varphi}_{\varepsilon e}\colon \overline{H}_{e} \to \overline{G}_{e(\varepsilon)}$, переводящее смежный класс~$\bar h$ ($h \in H_{e}$) в~элемент~$(h\varphi_{\varepsilon e})R_{e(\varepsilon)}$, корректно определено и~является изоморфизмом групп~$\overline{H}_{e}$ и~$H_{\varepsilon e}R_{e(\varepsilon)}/R_{e(\varepsilon)}$. Поэтому можно рассмотреть граф групп~$\mathcal{G}_{\mathcal{R}}(\mathcal{T})$, в~котором вершинам сопоставлены группы~$\overline{G}_{v}$ ($v \in \mathcal{V}$), а~ребрам~--- группы~$\overline{H}_{e}$ и~изоморфизмы~$\overline{\varphi}_{\pm e}$ ($e \in \mathcal{E}$). Легко видеть, что тождественное отображение образующих группы~$\pi_{1}(\mathcal{G}(\mathcal{T}))$ в~группу~$\pi_{1}(\mathcal{G}_{\mathcal{R}}(\mathcal{T}))$ определяет сюръективный гомоморфизм $\rho_{\mathcal{R}}\colon \pi_{1}(\mathcal{G}(\mathcal{T})) \to \pi_{1}(\mathcal{G}_{\mathcal{R}}(\mathcal{T}))$, ядро которого совпадает с~нормальным замыканием множества $\bigcup_{v \in \mathcal{V}}R_{v}$ в~группе~$\pi_{1}(\mathcal{G}(\mathcal{T}))$. Из~определения графа групп~$\mathcal{G}_{\mathcal{R}}(\mathcal{T})$ следует также, что гомоморфизм~$\rho_{\mathcal{R}}$ переводит вершинные группы и~реберные подгруппы группы~$\pi_{1}(\mathcal{G}(\mathcal{T}))$ на~соответствующие вершинные группы и~реберные подгруппы группы~$\pi_{1}(\mathcal{G}_{\mathcal{R}}(\mathcal{T}))$, продолжая естественные гомоморфизмы $G_{v} \to G_{v}/R_{v}$ ($v \in \mathcal{V}$).

\begin{rproposition}\label{rp52}
Пусть $\mathcal{N}$~--- предкорневой класс нильпотентных групп\textup{,} замкнутый относительно взятия фак\-тор-групп\textup{,} и~$\mathcal{N}^{\mathfrak{tf}}$~--- класс всех $\mathcal{N}$\nobreakdash-групп без~кручения. Пусть также $\mathcal{T}$~--- конечное дерево с~$n \geqslant 1$ вершинами\textup{,} все группы~$G_{v}$ \textup{(}$v \in \mathcal{V}$\textup{)} принадлежат классу~$\mathcal{N}^{\mathfrak{tf}}$ и~все подгруппы~$H_{\varepsilon e}$ \textup{(}$e \in \mathcal{E}$\textup{,} $\varepsilon = \pm 1$\textup{)} являются нетривиальными апериодическими локально циклическими группами. Тогда справедливы следующие утверждения.

\textup{1.}\hspace{1ex}Существует гомоморфизм группы~$\pi_{1}(\mathcal{G}(\mathcal{T}))$ на~$(\mathit{Poly}_{\max\{1,n-1\}}\mathcal{N}^{\mathfrak{tf}})$-груп\-пу\textup{,} действующий инъективно на~всех подгруппах~$H_{\varepsilon e}$ \textup{(}$e \in \mathcal{E}$\textup{,} $\varepsilon = \pm 1$\textup{)}.

\textup{2.}\hspace{1ex}Для~любых вершины $u \in \mathcal{V}$ и~подгруппы $K \leqslant G_{u}$ существует гомоморфизм группы~$\pi_{1}(\mathcal{G}(\mathcal{T}))$ на~$(\mathit{Poly}_{n}\mathcal{N}^{\mathfrak{tf}})$-груп\-пу\textup{,} действующий инъективно на~подгруппах~$H_{\varepsilon e}$ \textup{(}$e \in \mathcal{E}$\textup{,} $\varepsilon = \pm 1$\textup{)} и~подгруппе~$K$.
\end{rproposition}

\begin{proof}
Будем проверять оба утверждения одновременно, пользуясь индукцией по~$n$. Если $n = 1$, искомым гомоморфизмом является тождественное отображение единственной вершинной группы. Поэтому далее будем считать, что $n > 1$. В~силу предложения~\ref{rp42} каждой подгруппе~$H_{\varepsilon e}$ ($e \in \mathcal{E}$, $\varepsilon = \pm 1$) соответствует число $\lambda_{\varepsilon e} \geqslant 1$ такое, что $H_{\varepsilon e} \leqslant \zeta_{\lambda_{\varepsilon e}}G_{e(\varepsilon)}$ и~$H_{\varepsilon e} \cap \zeta_{\lambda_{\varepsilon e}-1}G_{e(\varepsilon)} = 1$. Для~любой вершины $v \in \mathcal{V}$ положим $\lambda_{v} = \min\{\lambda_{\varepsilon e} \mid (e,\varepsilon) \in \Theta(v)\} - 1$ и~$\mu_{\varepsilon e} = \lambda_{\varepsilon e} - \lambda_{v}$ ($(e,\varepsilon) \in \Theta(v)$). Так как по~предложению~\ref{rp41} подгруппа~$\zeta_{\lambda_{v}}G_{v}$ изолирована в~группе~$G_{v}$ и~класс~$\mathcal{N}$ замкнут относительно взятия фак\-тор-групп, то~$G_{v}/\zeta_{\lambda_{v}}G_{v} \in \mathcal{N}^{\mathfrak{tf}}$. Из~определения числа~$\lambda_{v}$ следует, что естественный гомоморфизм $\eta_{v}\colon G_{v} \to G_{v}/\zeta_{\lambda_{v}}G_{v}$ действует инъективно на~всех подгруппах~$H_{\varepsilon e}$ ($(e,\varepsilon) \in \Theta(v)$) и~хотя~бы для~одной пары $(e,\varepsilon) \in \Theta(v)$ имеет место равенство $\mu_{\varepsilon e} = 1$, означающее, что подгруппа~$H_{\varepsilon e}\eta_{v}$ лежит в~центре группы~$G_{v}\eta_{v}$. Покажем, что найдется ребро $f \in \mathcal{E}$ такое, что $\mu_{+f} = 1 = \mu_{-f}$.

В~самом деле, превратим дерево~$\mathcal{T}$ в~ориентированный граф~$\overset{\kern.4ex{}\rightarrow}{\mathcal{T}}$ следующим образом. Если $e \in \mathcal{E}$ и~$\mu_{\varepsilon e} < \mu_{-\varepsilon e}$ для~некоторого $\varepsilon = \pm 1$, то~будем считать вершину~$e(\varepsilon)$ началом, а~вершину~$e(-\varepsilon)$~--- концом дуги~$e$. Если~же $\mu_{\varepsilon e} = \mu_{-\varepsilon e}$, ориентируем дугу~$e$ произвольным образом. Так как граф~$\overset{\kern.4ex{}\rightarrow}{\mathcal{T}}$ конечен и~не~содержит замкнутых путей, \pagebreak то,~начав с~некоторой вершины $v \in \mathcal{V}$ и~двигаясь произвольным образом вдоль дуг графа, неизбежно придем в~ка\-кую-то вершину $v^{\prime} \in \mathcal{V}$, из~которой не~выходит ни~одной дуги. Но,~как было отмечено выше, хотя~бы для~одной пары $(f,\varepsilon) \in \Theta(v^{\prime})$ имеет место равенство $\mu_{\varepsilon f} = 1$. Значит, $\mu_{-\varepsilon f} = \mu_{\varepsilon f} = 1$ и~$f$~--- искомое ребро.

Так как каждая подгруппа $R_{v} = \ker\eta_{v}$ ($v \in \mathcal{V}$) тривиально пересекается со~всеми подгруппами~$H_{\varepsilon e}$ ($(e,\varepsilon) \in \Theta(v)$), то~семейство $\mathcal{R} = \{R_{v} \mid v \in \mathcal{V}\}$ является системой совместимых нормальных подгрупп. Поэтому определены граф групп~$\mathcal{G}_{\mathcal{R}}(\mathcal{T})$ и~гомоморфизм $\rho_{\mathcal{R}}\colon \pi_{1}(\mathcal{G}(\mathcal{T})) \to \pi_{1}(\mathcal{G}_{\mathcal{R}}(\mathcal{T}))$, продолжающий гомоморфизмы~$\eta_{v}$ ($v \in \mathcal{V}$) и,~следовательно, инъективный на~всех подгруппах~$H_{\varepsilon e}$ ($e \in \mathcal{E}$, $\varepsilon = \pm 1$). Обозначим через~$\mathcal{T}_{f}$ поддерево дерева~$\mathcal{T}$, состоящее из~вершин~$f(1)$,~$f(-1)$ и~ребра~$f$. Понятно, что группу~$\pi_{1}(\mathcal{G}_{\mathcal{R}}(\mathcal{T}))$ можно рассматривать как фундаментальную группу графа групп~$\mathcal{G}_{\mathcal{R}}^{\prime}(\mathcal{T}^{\prime})$, получающегося из~графа~$\mathcal{G}_{\mathcal{R}}(\mathcal{T})$ путем удаления ребра~$f$, склеивания его концов и~сопоставления этой новой вершине~$w$ группы~$\pi_{1}(\mathcal{G}_{\mathcal{R}}(\mathcal{T}_{f}))$. Положим $\mathcal{V}^{\prime} = \mathcal{V} \setminus \{f(1),f(-1)\} \cup \{w\}$.

Так как в~силу выбора ребра~$f$ подгруппы~$H_{+f}\rho_{\mathcal{R}}$ и~$H_{-f}\rho_{\mathcal{R}}$ лежат в~центрах групп $G_{f(1)}\rho_{\mathcal{R}}$ и~$G_{f(-1)}\rho_{\mathcal{R}}$ соответственно, то~согласно предложению~\ref{rp48} группа~$\pi_{1}(\mathcal{G}_{\mathcal{R}}(\mathcal{T}_{f}))$ обладает гомоморфизмом~$\psi$ на~$\mathcal{N}^{\mathfrak{tf}}$\nobreakdash-груп\-пу, действующим инъективно на~подгруппах~$G_{f(1)}\rho_{\mathcal{R}}$ и~$G_{f(-1)}\rho_{\mathcal{R}}$, а,~значит, и~на~всех реберных подгруппах графов~$\mathcal{G}_{\mathcal{R}}(\mathcal{T})$ и~$\mathcal{G}_{\mathcal{R}}^{\prime}(\mathcal{T}^{\prime})$, содержащихся в~этих группах. Положим $S_{w} = \ker\psi$ и~$S_{v} = 1$ для~всех $v \in \mathcal{V}^{\prime} \setminus \{w\}$. Тогда семейство $\mathcal{S} = \{S_{v} \mid v \in \mathcal{V}^{\prime}\}$ оказывается системой совместимых нормальных подгрупп в~графе групп~$\mathcal{G}_{\mathcal{R}}^{\prime}(\mathcal{T}^{\prime})$ и~потому определены граф групп~$(\mathcal{G}_{\mathcal{R}}^{\prime})_{\mathcal{S}}^{\vphantom{\prime}}(\mathcal{T}^{\prime})$ и~гомоморфизм $\rho^{\prime}_{\mathcal{S}}\colon \pi_{1}(\mathcal{G}_{\mathcal{R}}^{\prime}(\mathcal{T}^{\prime})) \to \pi_{1}((\mathcal{G}_{\mathcal{R}}^{\prime})_{\mathcal{S}}^{\vphantom{\prime}}(\mathcal{T}^{\prime}))$, продолжающий~$\psi$ и~тождественные отображения групп~$G_{v}\rho_{\mathcal{R}}$ ($v \in \mathcal{V}^{\prime} \setminus \{w\}$). Последнее означает, что гомоморфизм~$\rho^{\prime}_{\mathcal{S}}$ инъективен на~всех подгруппах~$H_{\varepsilon e}\rho_{\mathcal{R}}$ ($e \in \mathcal{E}$, $\varepsilon = \pm 1$).

Так как группа~$\pi_{1}((\mathcal{G}_{\mathcal{R}}^{\prime})_{\mathcal{S}}^{\vphantom{\prime}}(\mathcal{T}^{\prime}))$ удовлетворяет условиям настоящего предложения, а~дерево~$\mathcal{T}^{\prime}$ имеет $n-1$~вершину, то~к~указанной группе, вершине~$w$ и~подгруппе $K^{\prime} = (H_{+f}\rho_{\mathcal{R}}^{\vphantom{\prime}})\rho^{\prime}_{\mathcal{S}} = (H_{-f}\rho_{\mathcal{R}}^{\vphantom{\prime}})\rho^{\prime}_{\mathcal{S}}$ можно применить индуктивное предположение. Согласно ему группа~$\pi_{1}((\mathcal{G}_{\mathcal{R}}^{\prime})_{\mathcal{S}}^{\vphantom{\prime}}(\mathcal{T}^{\prime}))$ обладает гомоморфизмом~$\sigma$ на~группу из~класса~$\mathit{Poly}_{n-1}\mathcal{N}^{\mathfrak{tf}}$, инъективным на~подгруппах~$H_{\varepsilon e}\rho_{\mathcal{R}}^{\vphantom{\prime}}\rho^{\prime}_{\mathcal{S}}$ ($e \in \mathcal{E} \setminus \{f\}$, $\varepsilon = \pm 1$) и~подгруппе~$K^{\prime}$. Следовательно, композиция~$\rho_{\mathcal{R}}^{\vphantom{\prime}}\rho^{\prime}_{\mathcal{S}}\sigma$ действует инъективно на~всех подгруппах~$H_{\varepsilon e}$ ($e \in \mathcal{E}$, $\varepsilon = \pm 1$), что доказывает утверждение~1. Отсюда и~из~предложения~\ref{rp51} вытекает утверждение~2.
\end{proof}

\begin{proof}[\textbf{\textup{Доказательство утверждения~$\boldsymbol{(\alpha)}$}}]
Для~каждой вершины $v \in \mathcal{V}$ обозначим через~$R_{v}$ ядро естественного гомоморфизма $\eta_{v}\colon G_{v} \to G_{v}/\mathfrak{T}(G_{v})$. Так как подгруппы~$H_{\varepsilon e}$ ($e \in \mathcal{E}$, $\varepsilon = \pm 1$) не~имеют кручения, то~$H_{\varepsilon e} \cap R_{e(\varepsilon)} = 1$ для~всех $e \in \mathcal{E}$, $\varepsilon = \pm 1$. Поэтому семейство $\mathcal{R} = \{R_{v} \mid v \in \mathcal{V}\}$ является системой совместимых нормальных подгрупп, определены граф групп~$\mathcal{G}_{\mathcal{R}}(\mathcal{T})$ и~гомоморфизм $\rho_{\mathcal{R}}\colon \pi_{1}(\mathcal{G}(\mathcal{T})) \to \pi_{1}(\mathcal{G}_{\mathcal{R}}(\mathcal{T}))$. Последний продолжает отображения~$\eta_{v}$ ($v \in \mathcal{V}$) и,~следовательно, действует инъективно на~подгруппах~$H_{\varepsilon e}$ ($e \in \mathcal{E}$, $\varepsilon = \pm 1$). Так как $G_{v}\rho_{\mathcal{R}} = G_{v}/\mathfrak{T}(G_{v}) \in \mathcal{N}^{\mathfrak{tf}}$ для~всех $v \in \mathcal{V}$, то~в~силу предложения~\ref{rp52} гомоморфизм~$\rho_{\mathcal{R}}$ можно продолжить до~\mbox{искомого}.
\end{proof}

\section{Доказательство утверждений \texorpdfstring{$(\beta)$}{(beta)}, \texorpdfstring{$(\gamma)$}{(gamma)} и~следствия}\label{rsec06}

Пусть $P$~--- свободное произведение групп~$X_{i}$ ($i \in \mathcal{I}$). Напомним, что \emph{нормальной формой} элемента $x \in P$ называется запись этого элемента в~виде произведения $x_{1}x_{2} \ldots x_{l}$, где $l \geqslant 0$, $x_{1},\,x_{2},\,\ldots,\,x_{l} \in \bigcup_{i \in \mathcal{I}}X_{i} \setminus \{1\}$ и,~если $l \geqslant 2$, то~для~любого $k \in \{1,\,2,\,\ldots,\,l-1\}$ элементы~$x_{k}$ и~$x_{k+1}$ содержатся в~разных свободных множителях~$X_{i}$ ($i \in \mathcal{I}$) (предполагается, что если $l = 0$, то~$x_{1}x_{2} \ldots x_{l} = 1$). Пользуясь теоремой~1.2 из~\cite[гл.~IV]{LyndonSchupp1980} и~индукцией по~числу свободных множителей, легко доказать, что каждый элемент $x \in P$ обладает единственной нормальной формой и~потому $x \ne 1$ тогда и~только тогда, когда длина~$l$ этой формы отлична~от~$0$.

\begin{rproposition}\label{rp61}
Пусть $\mathcal{C}$~--- предкорневой класс групп\textup{,} $\Phi$~--- класс всех свободных групп\textup{,} $n \geqslant 1$ и~$P$~--- свободное произведение $(\mathit{Poly}_{n}\mathcal{C})$-ап\-прок\-си\-ми\-руе\-мых групп~$X_{i}$\linebreak \textup{(}$i \in \mathcal{I}$\textup{)} и~некоторой свободной группы~$Y$. Тогда группа~$P$ аппроксимируется классом~$\Phi \cdot \mathit{Poly}_{n}\mathcal{C}$.
\end{rproposition}

\begin{proof}
Пусть $x \in P \setminus \{1\}$~--- произвольный элемент и~$x_{1}x_{2} \ldots x_{l}$~--- его нормальная форма. Тогда $l > 0$ и~$x_{1},\,x_{2},\,\ldots,\,x_{l} \in \bigcup_{j \in \mathcal{J}}X_{j} \cup Y \setminus \{1\}$ для~некоторого конечного подмножества~$\mathcal{J}$ множества~$\mathcal{I}$. Так как группы~$X_{j}$ ($j \in \mathcal{J}$) $(\mathit{Poly}_{n}\mathcal{C})$-ап\-прок\-си\-ми\-руе\-мы и~класс~$\mathit{Poly}_{n}\mathcal{C}$ замкнут относительно взятия подгрупп и~конечных прямых произведений, то~согласно предложению~\ref{rp43} для~каждого $j \in \mathcal{J}$ множество $\{x_{k} \mid 1 \leqslant k \leqslant l\} \cap X_{j}$ не~пересекается с~подходящим образом выбранной подгруппой $R_{j} \in (\mathit{Poly}_{n}\mathcal{C})^{*}(X_{j})$. Рассмотрим гомоморфизм~$\sigma$ группы~$P$ на~свободное произведение групп~$X_{j}/R_{j}$ ($j \in \mathcal{J}$) и~группы~$Y$, переводящий группы~$X_{i}$ ($i \in \mathcal{I} \setminus \mathcal{J}$) в~единицу и~продолжающий тождественные отображения образующих групп~$X_{j}$ ($j \in \mathcal{J}$) и~$Y$.{\parfillskip=0pt{}\par}

Ввиду выбора подгрупп~$R_{j}$ ($j \in \mathcal{J}$) произведение $x_{1}\sigma x_{2}\sigma \ldots x_{l}\sigma$ представляет собой нормальную форму элемента~$x\sigma$ длины $l > 0$. Следовательно, $x\sigma \ne 1$. Пусть $\tau$~--- гомоморфизм группы~$P\sigma$ на~прямое произведение~$D$ групп~$X_{j}/R_{j}$ ($j \in \mathcal{J}$), продолжающий тождественные отображения сомножителей~$X_{j}/R_{j}$ ($j \in \mathcal{J}$) и~переводящий~$Y$ в~$1$. Так как $\ker\tau \cap z^{-1}(X_{j}/R_{j})z = z^{-1}(\ker\tau \cap (X_{j}/R_{j}))z = 1$ для~всех $j \in \mathcal{J}$, $z \in P\sigma$, то~согласно теореме Куроша о~строении подгрупп свободного произведения~\cite{Kurosch1934MA} подгруппа~$\ker\tau$ раскладывается в~свободное произведение некоторой свободной подгруппы и~подгрупп вида~$\ker\tau \cap z^{-1}Yz$ ($z \in P\sigma$). Поскольку $Y$~--- свободная группа, подгруппа~$\ker\tau$ также оказывается свободной. Отсюда и~из~включения $D \in \mathit{Poly}_{n}\mathcal{C}$ следует, что $P\sigma \in \Phi \cdot \mathit{Poly}_{n}\mathcal{C}$. Таким образом, ввиду произвольности выбора элемента~$x$ группа~$P$ $(\Phi \cdot \mathit{Poly}_{n}\mathcal{C})$-ап\-прок\-си\-ми\-руе\-ма.
\end{proof}

\begin{rproposition}\label{rp62}
Пусть $\mathcal{C}$~--- предкорневой класс групп\textup{,} дерево~$\mathcal{T}$ конечно и~для~не\-ко\-то\-ро\-го $n \geqslant 1$ существует гомоморфизм~$\sigma$ группы~$\pi_{1}(\mathcal{G}(\mathcal{T}))$ на~группу из~класса~$\mathit{Poly}_{n}\mathcal{C}$\textup{,} действующий инъективно на~всех подгруппах~$H_{\varepsilon e}$ \textup{(}$e \in \mathcal{E}$\textup{,} $\varepsilon = \pm 1$\textup{)}. Если все группы~$G_{v}$ \textup{(}$v \in \mathcal{V}$\textup{)} аппроксимируются классом~$\mathit{Poly}_{m}\mathcal{C}$ и~все свободные группы аппроксимируются классом~$\mathit{Poly}_{l}\mathcal{C}$ для~некоторых $l,m \geqslant 1$\textup{,} то~группа~$\pi_{1}(\mathcal{G}(\mathcal{T}))$ аппроксимируется классом~$\mathit{Poly}_{l+m+n}\mathcal{C}$.
\end{rproposition}

\begin{proof}
Пусть $N = \ker\sigma$. Согласно предложению~\ref{rp46} подгруппа~$N$ раскладывается в~свободное произведение некоторой свободной подгруппы и~подгрупп вида $x^{-1}(N \cap G_{v})x$, где $v \in \mathcal{V}$ и~$x \in \pi_{1}(\mathcal{G}(\mathcal{T}))$. Так как группы~$G_{v}$ ($v \in \mathcal{V}$) $(\mathit{Poly}_{m}\mathcal{C})$-ап\-прок\-си\-ми\-руе\-мы и~класс~$\mathit{Poly}_{m}\mathcal{C}$ замкнут относительно взятия подгрупп и~конечных прямых произведений, то~согласно предложению~\ref{rp43} подгруппы~$N \cap G_{v}$ ($v \in \mathcal{V}$) тоже являются $(\mathit{Poly}_{m}\mathcal{C})$-ап\-прок\-си\-ми\-руе\-мы\-ми. Отсюда и~из~предложения~\ref{rp61} следует, что группа~$N$ аппроксимируется классом~$\Phi \cdot \mathit{Poly}_{m}\mathcal{C}$, где $\Phi$~--- класс всех свободных групп. В~силу предложения~\ref{rp32} каждая $(\Phi \cdot \mathit{Poly}_{m}\mathcal{C})$-груп\-па является $(\mathit{Poly}_{l+m}\mathcal{C})$-ап\-прок\-си\-ми\-руе\-мой. Значит, подгруппа~$N$ тоже $(\mathit{Poly}_{l+m}\mathcal{C})$-ап\-прок\-си\-ми\-руе\-ма и~группа~$\pi_{1}(\mathcal{G}(\mathcal{T}))$ аппроксимируется классом~$\mathit{Poly}_{l+m+n}\mathcal{C}$ ввиду того~же предложения~\ref{rp32}.
\end{proof}

\begin{proof}[\textbf{\textup{Доказательство утверждений~$\boldsymbol{(\beta)}$ и~$\boldsymbol{(\gamma)}$}}]
Согласно предложению~\ref{rp44}, если группа~$G_{e(\varepsilon)}$ ($e \in \mathcal{E}$, $\varepsilon = \pm 1$) аппроксимируется классом~$\mathcal{N}^{\mathfrak{tf}}$, то~она обладает гомоморфизмом на~$\mathcal{N}^{\mathfrak{tf}}$\nobreakdash-груп\-пу, инъективным на~подгруппе~$H_{\varepsilon e}$. Поэтому утверждения~$(\beta)$ и~$(\gamma)$ можно доказывать одновременно, считая далее, что символ~$\mathcal{C}$ обозначает класс~$\mathcal{N}$ (в~случае утверждения~$(\beta)$) или~класс~$\mathcal{N}^{\mathfrak{tf}}$ (в~случае утверждения~$(\gamma)$).

Из~предложения~\ref{rp43} следует, что для~каждой вершины $v \in \mathcal{V}$ существует гомоморфизм~$\sigma_{v}$ группы~$G_{v}$ на~$\mathcal{C}$\nobreakdash-груп\-пу, инъективный на~всех подгруппах~$H_{\varepsilon e}$ ($(e,\varepsilon) \in\nolinebreak \Theta(v)$). Семейство $\mathcal{R} = \{\ker\sigma_{v} \mid v \in \mathcal{V}\}$ является, очевидно, системой совместимых нормальных подгрупп в~графе~$\mathcal{G}(\mathcal{T})$. Поэтому определены граф групп~$\mathcal{G}_{\mathcal{R}}(\mathcal{T})$ и~гомоморфизм $\rho_{\mathcal{R}}\colon \pi_{1}(\mathcal{G}(\mathcal{T})) \to \pi_{1}(\mathcal{G}_{\mathcal{R}}(\mathcal{T}))$, продолжающий гомоморфизмы~$\sigma_{v}$ ($v \in \mathcal{V}$) и~переводящий реберные подгруппы группы~$\pi_{1}(\mathcal{G}(\mathcal{T}))$ на~реберные подгруппы группы~$\pi_{1}(\mathcal{G}_{\mathcal{R}}(\mathcal{T}))$. Поскольку группа~$\pi_{1}(\mathcal{G}_{\mathcal{R}}(\mathcal{T}))$ удовлетворяет условиям утверждения~$(\alpha)$, найдется гомоморфизм~$\tau$ этой группы, отображающий ее на~группу из~класса~$\mathit{Poly}_{n-1}\mathcal{N}^{\mathfrak{tf}}$ и~инъективный на~всех подгруппах~$H_{\varepsilon e}\rho_{\mathcal{R}}$ ($e \in \mathcal{E}$, $\varepsilon = \pm 1$). Тогда композиция~$\rho_{\mathcal{R}}\tau$ действует инъективно на~подгруппах~$H_{\varepsilon e}$ ($e \in \mathcal{E}$, $\varepsilon = \pm 1$) и~согласно предложению~\ref{rp62} группа~$\pi_{1}(\mathcal{G}(\mathcal{T}))$ аппроксимируется классом~$\mathit{Poly}_{n+1}\mathcal{C}$.
\end{proof}

\begin{proof}[\textbf{\textup{Доказательство следствия}}]
$(\dag)$\hspace{1ex}Так как группы~$G_{v}$ ($v \in \mathcal{V}$) $\mathcal{N}^{\mathfrak{tf}}$\nobreakdash-ап\-прок\-си\-ми\-руе\-мы, они не~имеют кручения и~тем~же свойством обладают подгруппы~$H_{\varepsilon e}$ ($e \in \mathcal{E}$, $\varepsilon = \pm 1$). Обозначая через~$\mathcal{N}$ класс всех нильпотентных групп и~используя ту~же аргументацию, что~и~при~доказательстве утверждений $(\beta)$ и~$(\gamma)$, получаем, что существует гомоморфизм группы~$\pi_{1}(\mathcal{G}(\mathcal{T}))$, отображающий ее на~$(\mathit{Poly}_{n-1}\mathcal{N}^{\mathfrak{tf}})$-груп\-пу и~действующий инъективно на~всех реберных подгруппах. Согласно предложению~\ref{rp46} ядро~$K$ этого гомоморфизма раскладывается в~свободное произведение некоторой свободной подгруппы и~подгрупп вида $x^{-1}(K \cap G_{v})x$, где $v \in \mathcal{V}$ и~$x \in \pi_{1}(\mathcal{G}(\mathcal{T}))$. Хорошо известно, что всякая свободная группа $\mathcal{N}^{\mathfrak{tf}}$\nobreakdash-ап\-прок\-си\-ми\-руе\-ма~\cite{Hall1950PAMS, Magnus1935MA} и~свободное произведение любого числа $\mathcal{N}^{\mathfrak{tf}}$\nobreakdash-ап\-прок\-си\-ми\-руе\-мых групп, в~свою очередь, аппроксимируется классом~$\mathcal{N}^{\mathfrak{tf}}$~\cite{Malcev1949MS}. Отсюда, из~$\mathcal{N}^{\mathfrak{tf}}$\nobreakdash-ап\-прок\-си\-ми\-руе\-мо\-сти групп~$G_{v}$ ($v \in \mathcal{V}$) и~предложения~\ref{rp43} вытекает, что все подгруппы вида $x^{-1}(K \cap G_{v})x$ ($v \in \mathcal{V}$, $x \in \pi_{1}(\mathcal{G}(\mathcal{T}))$) $\mathcal{N}^{\mathfrak{tf}}$\nobreakdash-ап\-прок\-си\-ми\-руе\-мы и~древесное произведение~$\pi_{1}(\mathcal{G}(\mathcal{T}))$ является расширением $\mathcal{N}^{\mathfrak{tf}}$\nobreakdash-ап\-прок\-си\-ми\-руе\-мой группы~$K$ при~помощи $(\mathit{Poly}_{n-1}\mathcal{N}^{\mathfrak{tf}})$-груп\-пы. Значит, указанное произведение аппроксимируется классом~$\mathit{Poly}_{n}\mathcal{N}^{\mathfrak{tf}}$ в~силу предложения~\ref{rp32}.

$(\ddag)$\hspace{1ex}Пусть $x \in \pi_{1}(\mathcal{G}(\mathcal{T})) \setminus \{1\}$~--- произвольный элемент и~$B_{v}$ ($v \in \mathcal{V}$)~--- некоторые свободные базисы групп~$G_{v}$ ($v \in \mathcal{V}$). Так как всякая локально циклическая подгруппа свободной группы является циклической, то~группа~$\pi_{1}(\mathcal{G}(\mathcal{T}))$ имеет циклические реберные подгруппы и~потому может быть задана представлением с~множеством образующих $S = \bigsqcup_{v \in \mathcal{V}}B_{v}$ и~некоторым конечным множеством определяющих соотношений. Отсюда следует, что существует конечное подмножество $Q \subseteq S$, содержащее все образующие из~определяющих соотношений группы~$\pi_{1}(\mathcal{G}(\mathcal{T}))$ и~некоторого фиксированного слова, задающего элемент~$x$.

Очевидно, что группа~$\pi_{1}(\mathcal{G}(\mathcal{T}))$ раскладывается в~свободное произведение подгрупп $F_{1} = \operatorname{sgp}\{Q\}$ и~$F_{2} = \operatorname{sgp}\{S \setminus Q\}$, причем подгруппа~$F_{1}$ является древесным произведением $n$~свободных групп конечного ранга с~нетривиальными циклическими реберными подгруппами и~$x \in F_{1} \setminus \{1\}$. Как уже было упомянуто выше, всякая свободная группа $\mathcal{N}^{\mathfrak{tf}}$\nobreakdash-ап\-прок\-си\-ми\-руе\-ма. Поэтому группа~$F_{1}$ аппроксимируется классом~$\mathit{Poly}_{n}\mathcal{N}^{\mathfrak{tf}}$ согласно утверждению~$(\dag)$. Так~как она конечно порождена, то~тем~же свойством обладают и~все ее гомоморфные образы. Значит, гомоморфизм $\rho\colon \pi_{1}(\mathcal{G}(\mathcal{T})) \to F_{1}$, действующий тождественно на~$F_{1}$ и~отправляющий сомножитель~$F_{2}$ в~$1$, можно продолжить до~гомоморфизма на~конечно порожденную $(\mathit{Poly}_{n}\mathcal{N}^{\mathfrak{tf}})$\nobreakdash-груп\-пу, переводящего~$x$ в~неединичный элемент.
\end{proof}